\documentclass[a4paper,12pt,oneside]{article}
\usepackage{amsmath,amsfonts,amssymb,amsthm}
\usepackage{bbm}

\theoremstyle{plain}
\newtheorem{theorem}{Theorem}[section]
\newtheorem{corollary}[theorem]{Corollary}
\newtheorem{lemma}[theorem]{Lemma}
\newtheorem{proposition}[theorem]{Proposition}

\theoremstyle{definition}
\newtheorem{definition}[theorem]{Definition}

\newtheorem{example}[theorem]{Example}

\numberwithin{equation}{section}

\newcommand{\R}{{\mathbb R}}
\newcommand{\N}{{\mathbb N}}

\providecommand{\vint}[1]{\mathchoice
          {\mathop{\vrule width 5pt height 3 pt depth -2.5pt
                  \kern -9pt \kern 1pt\intop}\nolimits_{\kern -5pt{#1}}}
          {\mathop{\vrule width 5pt height 3 pt depth -2.6pt
                  \kern -6pt \intop}\nolimits_{\kern -3pt{#1}}}
          {\mathop{\vrule width 5pt height 3 pt depth -2.6pt
                  \kern -6pt \intop}\nolimits_{\kern -3pt{#1}}}
          {\mathop{\vrule width 5pt height 3 pt depth -2.6pt
                  \kern -6pt \intop}\nolimits_{\kern -3pt{#1}}}}

\newcommand{\eps}{\varepsilon}

\newcommand{\BV}{\mathrm{BV}}
\renewcommand{\L}{\mathrm{L}}

\newcommand{\id}{\mathrm{id}}
\newcommand{\mres}{\!\mathbin{\vrule height 1.6ex depth 0pt width
0.13ex\vrule height 0.13ex depth 0pt width 1.1ex}\!}

\DeclareMathOperator{\dist}{dist}

\DeclareMathOperator{\esssup}{ess\,sup}

\makeatletter
\def\blfootnote{\xdef\@thefnmark{}\@footnotetext}
\makeatother

\begin{document}
\title{Lower semicontinuity for an integral\\
functional in $\BV$
}
\author{Jan Kristensen and Panu Lahti\\
Mathematical Institute, University of Oxford,\\ Andrew Wiles Building,\\
Radcliffe Observatory Quarter, Woodstock Road,\\
Oxford, OX2 6GG.\\
E-mail: {\tt jan.kristensen@maths.ox.ac.uk,}\\
{\tt lahti@maths.ox.ac.uk}
}
\maketitle

\begin{abstract}
We prove a lower semicontinuity result for a functional of linear growth initially defined by
\[
\int_{\Omega}F\left(\frac{dDu}{d\mu}\right)\,d\mu
\]
for $u\in\BV(\Omega;\R^N)$ with $Du\ll \mu$. The positive Radon measure $\mu$ is only assumed to satisfy $\mathcal L^n\ll \mu$.
\end{abstract}

\noindent {\bf Acknowledgments:} This research was done while P.L. was visiting the Oxford Centre for Nonlinear Partial Differential Equations (OxPDE) from September 2014 to July 2016. During this time, P.L. was supported by Aalto University as well as the Finnish Cultural Foundation.

\newpage
\section{Introduction}

In this work we prove a lower semicontinuity result for a functional of linear growth initially defined in an open set $\Omega\subset\R^n$ by
\begin{equation}\label{eq:definition of functional in intro}
\int_{\Omega}F\left(\frac{dDu}{d\mu}\right)\,d\mu
\end{equation}
for $u\in\BV(\Omega;\R^N)$ with $Du\ll \mu$. The measure $\mu$ is merely assumed to be a positive finite Radon measure that satisfies $\mathcal L^n\ll \mu$, where $\mathcal L^n$ is the Lebesgue measure. For the integrand $F$ we need somewhat stronger assumptions, described in detail below. We refer the reader to Section \ref{sec:preliminaries} for precise notation and terminology.

With the choice $\mu=\mathcal L^n$, this type of result was derived in \cite{ADM} (see also \cite[Section 5.5]{AFP}), and later in \cite{FM} for integrands depending also on $x$ and $u$. The problem was studied without a nonnegativity assumption on $F$ in \cite{KR1}. These results relied mostly on blow-up techniques. The result in \cite{KR1} was generalized to $x$-dependent integrands in \cite[Theorem 10]{KR2}, relying on the theory of generalized Young measures, which were first introduced by DiPerna and Majda in \cite{DPM}. With a general measure $\mu$, the problem was studied in the case $p>1$ in \cite{ABF}, and also in \cite{JK}. In a more general setting of a metric measure space, the problem was studied in \cite{HKLL}.

We first show that the functional \eqref{eq:definition of functional in intro}, defined for general $u\in\BV(\Omega;\R^N)$ by relaxation, has an integral representation
\[
\int_{\Omega}F\left(\frac{dDu}{d\mu}\right)\,d\mu+\int_{\Omega}F^{\infty}\left(\frac{dD^{s,\mu}u}{d|D^{s,\mu}u|}\right)\,d|D^{s,\mu}u|,
\]
where $D^{s,\mu}u$ is the singular part of $Du$ with respect to $\mu$.
Here we require the integrand $F:\R^{N\times n}\to \R$ be nonnegative and quasiconvex, with linear growth $m|A|\le F(A)\le M(1+|A|)$ for some $0<m\le M$ and all $A\in \R^{N\times n}$, and with a continuous recession function $F^{\infty}$.

Our proof will rely heavily on the theory of generalized Young measures, particularly results derived in \cite{KR2}. Once we have the above integral representation, we can derive Jensen's inequalities for generalized Young measures with respect to $\mu$, as was done in \cite[Theorem 9]{KR2} with respect to the Lebesgue measure. By using these inequalities, we can then prove the following lower semicontinuity theorem (Theorem \ref{thm:lower semicontinuity for $x$-dependent integrands}) which is the main result of this work:

\begin{theorem}\label{thm:lower semicontinuity result in intro}
Let $\Omega\subset \R^n$ be a bounded Lipschitz domain with inner boundary normal $\nu_{\Omega}$,
let $\mu$ be a positive finite Radon measure on $\Omega$ with $\mathcal L^n\ll \mu$, and
let $F\colon\overline{\Omega}\times\R^{N\times n}\to \R$ be a $\mu\times\mathcal{B}(\R^{N\times n})$-measurable integrand with linear growth $0\le F(x,A)\le M(1+|A|)$ for some $M\ge 0$, a continuous recession function $F^{\infty}$, and such that $A\mapsto F(x,A)$ is quasiconvex for each fixed $x\in\overline{\Omega}$. Then the functional
\begin{align*}
\mathcal F(u):=&\int_{\Omega}F\left(x,\frac{dDu}{d\mu}\right)\,d\mu+\int_{\Omega}F^{\infty}\left(x,\frac{dD^{s,\mu}u}{d|D^{s,\mu}u|}\right)\,d|D^{s,\mu}u|\\
&+\int_{\partial\Omega}F^{\infty}\left(x,\frac{u}{|u|}\otimes \nu_{\Omega}\right)|u|\,d\mathcal H^{n-1}
\end{align*}
is weakly* sequentially lower semicontinuous in $\BV(\Omega;\R^N)$.
\end{theorem}

We remark that an easier proof is possible when the Radon--Nikodym derivative of $\mu$ with respect to Lebesgue measure is bounded, and hence that the main contribution
is the proof covering the general case. This proof seems to require the assumption about existence of a continuous recession function for the integrand $F$.

\section{Preliminaries}\label{sec:preliminaries}

\subsection{Notation}
For $N,n\in\N$, the matrix space $\R^{N\times n}$ will always be equipped with the Euclidean norm $|A|:=\left(\sum_{i=1}^N \sum_{j=1}^n A_j^i\right)^{1/2}$, where $i$ and $j$ are the row and column indices, respectively.
We denote by $B(x,r)$ the open ball in $\R^{n}$ with center $x$ and radius $r$. We denote by $\mathbb{B}^{n}$ the open unit ball in $\R^{n}$ and by $\partial \mathbb{B}^{n}$ the unit sphere. For $a\in\R^N$ and $b\in\R^n$, we can define the tensor product $a\otimes b =a b^T\in \R^{N\times n}$.

We denote the $n$-dimensional Lebesgue measure by $\mathcal L^n$ and the $s$-di\-men\-sio\-nal Hausdorff measure by $\mathcal H^s$.
Given any measure $\nu$, the restriction of $\nu$ to a set $A$ is denoted by  $\nu\mres A$, that is, $\nu\mres A(B)=\nu(A\cap B)$. The Borel $\sigma$-algebra on a set $E\subset \R^n$ is denoted by $\mathcal B(E)$. For open sets $\Omega,\Omega'\subset \R^n$, by $\Omega\Subset \Omega'$ we mean that $\overline{\Omega}\subset \Omega'$ and that $\overline{\Omega}$ is compact. We denote by $\mathbbm{1}_E$ the characteristic function of a set $E$.

If $X$ is a locally compact separable metric space (usually an open or closed subset of $\R^n$), let $C_c(X;\R^{l})$ be the space of continuous $\R^{l}$-valued functions with compact support in $X$ and let $C_0(X;\R^{l})$ be its completion with respect to the $\Vert \cdot\Vert_{\infty}$-norm, $l\in\N$.
We denote by $\mathcal M(X;\R^{l})$ the Banach space of vector-valued finite Radon measures, equipped with the \emph{total variation norm} $|\mu|(X)<\infty$. By the Riesz representation theorem, $\mathcal M(X;\R^{l})$ can be identified with the dual space of $C_0(X;\R^{l})$ through the duality pairing 
$\langle\phi,\mu\rangle:=\int_X \phi\cdot d\mu:=\sum_{i=1}^l \int_X \phi_i \,d\mu_i$. Thus weak* convergence $\mu_j\overset{*}{\rightharpoondown}\mu$ in $\mathcal M(X)$ means $\langle\phi,\mu_j\rangle\to \langle\phi,\mu\rangle$ for all $\phi\in C_0(X;\R^{l})$. We denote the set of positive measures and probability measures by $\mathcal M^+(X)$ and $\mathcal M^1(X)$, respectively.

For a vector-valued Radon measure $\gamma\in\mathcal M(X;\R^{l})$ and a positive Radon measure $\mu\in\mathcal M^+(X)$, we can write the \emph{Radon-Nikodym decomposition} $\gamma=\gamma^a+\gamma^s=\frac{d\gamma}{d\mu}\mu+\gamma^s$ of $\gamma$ with respect to $\mu$, where $\frac{d\gamma}{d\mu}\in L^1(X,\mu;\R^l)$.

We write
\[
\vint{\Omega}f\,d\mu:=\frac{1}{\mu(\Omega)}\int_{\Omega}f\,d\mu
\]
for integral averages (whenever they are defined).


For sets $E\subset \R^n$, $F\subset \R^l$ open or closed, a \emph{parametrized measure} $(\nu_x)_{x\in E}\subset \mathcal M(F)$ is a mapping from $E$ to the set $\mathcal M(F)$ of Radon measures on $F$. It is said to be \emph{weakly* $\mu$-measurable}, for $\mu\in\mathcal M^+(E)$, if $x\mapsto \nu_x(B)$ is $\mu$-measurable for all Borel sets $B\in\mathcal{B}(F)$ (it suffices to check this for all relatively open sets).
Equivalently, $(\nu_x)_{x\in E}$ is weakly* $\mu$-measurable if the function $x\mapsto \int_F f(x,y)\,d\nu_x(y)$ is $\mu$-measurable for every bounded Borel function $f:E\times F\to \R$ (see \cite[Proposition 2.26]{AFP}). We denote by $L^{\infty}_{w*}(E,\mu;\mathcal M(F))$ the set of all weakly* $\mu$-measurable parametrized measures $(\nu_x)_{x\in E}\subset \mathcal M(F)$ with the property that $\esssup_{x\in E}|\nu_x|(F)<\infty$ (the essential supremum with respect to $\mu$). We omit $\mu$ in the notation if $\mu=\mathcal L^n$.


\subsection{Functions of bounded variation}

The theory of $\BV$ functions presented in this section can be found in e.g. the monographs \cite{AFP,EvGa,Zie}, and we will give specific references only for a few key results.
Let $\Omega\subset \R^n$ be an open set. A function $u\in L^1(\Omega;\R^N)$ is a \emph{function of bounded variation}, denoted by $u\in\BV(\Omega;\R^N)$, if its distributional derivative is a bounded $\R^{N\times n}$-valued Radon measure. This means that there exists a (unique) measure $Du\in \mathcal M(\R^n;\R^{N\times n})$ such that for all $\psi\in C_c^1(\Omega)$, the integration-by-parts formula
\[
\int_{\Omega}\frac{\partial\psi}{\partial x_j}u^i\,d\mathcal L^n=-\int_{\Omega}\psi\,dDu_j^i, \qquad i=1\,\ldots N,\ \ j=1,\ldots,n
\]
holds. We write the Radon-Nikodym decomposition of the variation measure as $Du=\nabla u\,\mathcal L^n\mres \Omega+D^s u$.

The space $\BV(\Omega;\R^N)$ is a Banach space endowed with the norm
\[
\Vert u\Vert_{\BV(\Omega;\R^N)}:=\Vert u\Vert_{L^1(\Omega;\R^N)}+|Du|(\Omega).
\]
Furthermore, we say that a sequence $(u_j)\subset \BV(\Omega;\R^N)$ converges weakly* to $u\in\BV(\Omega;\R^N)$ if $u_j\to u$ strongly in $L^1(\Omega;\R^N)$ and $Du_j\overset{*}{\rightharpoondown}Du$ in $\mathcal M(\Omega,\R^{N\times n})$. A norm-bounded sequence in $\BV(\Omega;\R^N)$, i.e.
\[
\sup_{j\in\N}(\Vert u\Vert_{L^1(\Omega;\R^N)}+|Du_j|(\Omega))<\infty,
\]
always has a weakly* converging subsequence. Conversely, a weakly* converging sequence is norm-bounded in $\BV(\Omega,\R^N)$, see \cite[Proposition 3.13]{AFP}. If $u_j\to u$ in $L^1(\Omega;\R^N)$ and $|Du_j|(\Omega)\to |Du|(\Omega)$, we say that the $u_j$ converge to $u$ \emph{strictly}. If even
\[
\langle Du_j\rangle(\Omega)\to\langle Du\rangle(\Omega),
\]
where for a measure $\nu\in\mathcal M(\R^n;\R^{N\times n})$ with Radon-Nikodym decomposition $\nu=a\,\mathcal L^n+\mu^s$, we define the measure (related to the minimal surface functional)
\[
\langle\nu\rangle(A):=\int_A\sqrt{1+|a|^2}\,d\mathcal L^n+|\mu^s|(A),\qquad A\in \mathcal B(\R^n),
\]
then we speak of \emph{$\langle\cdot\rangle$-strict convergence}. This notion is stronger than strict convergence (this follows e.g. from Theorem \ref{thm:Reshetnyak} below), and one can show that it implies that $\langle Du_j\rangle\overset{*}{\rightharpoondown}\langle Du\rangle$ as measures.

For any bounded open set $\Omega\subset \R^n$ and $v\in\BV(\Omega;\R^N)$, we can define the Dirichlet class
\[
\BV_v(\Omega;\R^N):=\left\{u\in\BV(\Omega;\R^N):\,w\in\BV(\R^n;\R^N)\textrm{ and }|Dw|(\partial\Omega)=0\right\},
\]
where
\[
w:=
\begin{cases}
u-v &\textrm{in }\Omega,\\
0 &\textrm{in }\R^n\setminus\Omega.
\end{cases}
\]
The following lemma is proved in e.g. \cite[Lemma 1]{KR2}.

\begin{lemma}\label{lem:smooth area-strict approximation of BV}
Let $\Omega\subset \R^n$ be a bounded open set, and let $u\in\BV(\Omega;\R^N)$. Then there exists $(v_j)\subset \BV_u(\Omega;\R^N)\cap C^{\infty}(\Omega;\R^N)$ such that $v_j\to u$ $\langle\cdot\rangle$-strictly in $\Omega$.
\end{lemma}

\subsection{Generalized Young measures}\label{sec:generalized Young measures}

Most of the theory of generalized Young measures presented in this section is derived in \cite{KR2}.

The symbol $\Omega$ will always denote a bounded open set in $\R^n$.
We will need the following linear transformations mapping $C(\Omega\times\R^{l})$ to $C(\Omega\times\mathbb B^{l})$ and back, where $\mathbb B^{l}$ was the open unit ball in $\R^{l}$: for $f\in C(\Omega\times\R^{l})$ and $g\in C(\Omega\times\mathbb B^{l})$, define
\[
(Tf)(x,\hat{A}):=(1-|\hat{A}|)f\left(x,\frac{\hat{A}}{1-|\hat{A}|}\right),\ \ x\in\Omega,\, \ \hat{A}\in \mathbb B^{l},\quad\textrm{and}
\]
\[
(T^{-1}g)(x,A):=(1+|A|)g\left(x,\frac{A}{1+|A|}\right),\ \ x\in\Omega,\, \ A\in \R^{l}.
\]
It is an easy calculation to verify that $T^{-1}Tf=f$ and $T T^{-1}g=g$. We consider the property 
\begin{equation}\label{eq:definition of E}
Tf\textrm{ extends to a bounded continuous function on }\overline{\Omega\times\mathbb B^{l}}.
\end{equation}
In particular, this entails that $f$ has \emph{linear growth at infinity}, that is, there exists a constant $M\ge 0$ (in fact, $M=\Vert Tf\Vert_{L^{\infty}(\overline{\Omega\times \mathbb B^{l}})}$ will do) such that
\[
|f(x,A)|\le M(1+|A|)\quad\textrm{for all }x\in\overline{\Omega},\ A\in\R^{l}.
\]
We collect all such integrands into the set
\[
\mathbf{E}(\Omega;\R^{l}):=\{f\in C(\Omega\times\R^{l}):\, f\textrm{ satisfies }\eqref{eq:definition of E}\}.
\]
For $f\in\mathbf{E}(\Omega;\R^l)$, the \emph{recession function}
$f^{\infty}\colon\overline{\Omega}\times\R^{l}\mapsto \R$ is defined by
\begin{equation}\label{eq:definition of recession function}
f^{\infty}(x,A):=\lim_{\substack{x'\to x\\ A'\to A\\ t\to\infty}}\frac{f(x',tA')}{t},\quad x\in\overline{\Omega},\ A\in \R^{l}.
\end{equation}
The limit exists since it agrees with
$Tf$ on $\overline{\Omega}\times\partial\mathbb B^{N\times n}$, as
can be seen by substituting $t=s/(1-s)$, $s\in (0,1)$, and letting $s\to 1$.
The recession function is clearly positively $1$-homogenous in $A$, that is,
$f^{\infty}(x,sA)=s f^{\infty}(x,A)$ for all $s\ge 0$, and thus takes finite values. 

We also consider a second class of integrands that is larger than $\mathbf{E}(\Omega;\R^{l})$ and (partially) dispenses with continuity in the $x$-variable. A \emph{Car\-ath\-\'eo\-dory function} is an $\mathcal L^n\times \mathcal B(\R^{l})$-measurable function
$f\colon\overline{\Omega}\times \R^{l}\to \R$ such that $A\mapsto f(x,A)$ is continuous for almost every $x\in\overline{\Omega}$. In fact, it can be shown that it suffices to check measurability of $x\mapsto f(x,A)$ for all fixed $A\in \R^{l}$ (see for example \cite[Proposition 5.6]{AFP}). With this notion, the \emph{representation integrands} are defined as follows:
\begin{equation}\label{eq:definition of representation integrands}
\begin{split}
\mathbf{R}(\Omega;\R^{l}):=\{f\colon\overline{\Omega}\times\R^{l}\to \R:\ 
f \textrm{ Carath\'eodory with linear growth}&\\
 \textrm{at infinity and }\exists f^{\infty}\in C(\overline{\Omega}\times\R^{l})\}&.
\end{split}
\end{equation}

A function $f\colon\R^{N\times n}\to \R$ is said to be \emph{quasiconvex}, which we denote by $f\in\mathbf{Q}(\R^{N\times n})$, if $f$ is Borel measurable, locally bounded from below, and for some bounded Lipschitz domain $\omega\subset \R^n$ and every $A\in\R^{N\times n}$ it holds that
\[
|\omega|f(A)\le \int_{\omega}f(A+\nabla \psi(x))\,d\mathcal L^n(x)\quad \textrm{for all }\psi\in W_0^{1,\infty}(\omega;\R^N).
\]
This definition does not depend on the particular choice of the Lipschitz domain $\omega$ (by an exhaustion argument) and it can be shown that quasiconvex functions are \emph{rank one convex}, meaning that they are convex along rank one lines (see for example \cite[Proposition 5.41]{AFP}). See \cite{BD} for more on quasiconvexity.

A quasiconvex function does not necessarily have a recession function $f^{\infty}$ in the sense of \eqref{eq:definition of recession function} (see \cite[Theorem 2]{Mul} for a counterexample),
and the notion can be relaxed in the following way: for $f\colon\R^{N\times n}\to \R$ the \emph{generalized recession function} $f^{\#}\colon\R^{N\times n}\to \R\cup \{\pm\infty\}$ is defined by
\[
f^{\#}(A):=\limsup_{\substack{ A'\to A\\ t\to \infty}}\frac{f(tA')}{t},\quad A\in \R^{N\times n}.
\]
Quasiconvex functions are globally Lipschitz continuous (see for example \cite[Lemma 2.2]{BKK}) and hence for quasiconvex $f$
\begin{equation}\label{eq:definition of generalized recession function}
f^{\#}(A)=\limsup_{t\to\infty}\frac{f(tA)}{t},\quad A\in \R^{N\times n}.
\end{equation}
By rank one convexity, the above holds as a limit for all matrices $A$ of rank one.

We have the following version of Reshetnyak's Continuity Theorem, see the appendix of \cite{KR1}, as well as \cite[Theorem 3]{Resh} or \cite[Theorem 2.39]{AFP} for the original result stated for $1$-homogenous functions $f$.

\begin{theorem} \label{thm:Reshetnyak}
Let $(\gamma_j)\subset \mathcal M(\overline{\Omega};\R^{l})$, $\gamma \in\mathcal M(\overline{\Omega};\R^{l})$ with Radon-Nikodym decompositions
\[
\gamma_j=a_j\,\mathcal L^n\mres\Omega+\gamma_j^s,\qquad \gamma=a\,\mathcal L^n\mres\Omega+\gamma^s.
\]
If $\gamma_j\overset{*}{\rightharpoondown}\gamma$ in $\mathcal M(\overline{\Omega};\R^{l})$ and $\langle \gamma_j\rangle (\overline{\Omega})\to \langle \gamma \rangle (\overline{\Omega})$, then
for
\[
\mathcal F(\gamma):=\int_{\overline{\Omega}}f(x,a(x))\,d\mathcal L^n+\int_{\overline{\Omega}} f^{\infty}\left(x,\frac{d\gamma^s}{d|\gamma^s|}(x)\right)\,d|\gamma^s|(x)
\]
with $f\in\mathbf{E}(\Omega,\R^{l})$, we have
$\mathcal F(\gamma_j)\to \mathcal F(\gamma)$.
\end{theorem}

Let $\mu\in \mathcal M^+(\overline{\Omega})$, and assume that $\mu(\partial\Omega)=0$.

The set of all \emph{generalized Young measures} $\mathbf{Y}(\Omega,\mu;\R^{l})$ is defined to be the set of all triples $(\nu_x,\lambda_{\nu},\nu_x^{\infty})$ such that
\begin{alignat*}{2}
&(\nu_x)_x \in L_{w*}^{\infty}(\Omega,\mu;\mathcal M^1(\R^{l})),\qquad\qquad
&&\lambda_{\nu}\in\mathcal M^+(\overline{\Omega}),\\
&(\nu_x^{\infty})_x \in L_{w*}^{\infty}(\overline{\Omega},\lambda_{\nu};\mathcal M^1(\partial\mathbb B^{l})),
&&x\mapsto\langle |\cdot|,\nu_x \rangle\in \L^1(\Omega,\mu).
\end{alignat*}
Under the duality pairing
\begin{align*}
&\langle \! \langle f,\nu\rangle\! \rangle
:= \int_{\Omega}\langle f(x,\cdot),\nu_x\rangle\,d\mu(x)+\int_{\overline{\Omega}}\langle f^{\infty}(x,\cdot),\nu_x^{\infty}\rangle\,d\lambda_{\nu}(x)\\
&\quad =\int_{\Omega}\int_{\R^{N\times n}} f(x,A)\,d\nu_x(A)\,d\mu(x)+\int_{\overline{\Omega}}\int_{\partial\mathbb{ B}^{N\times n}}f^{\infty}(x,A)\,d\nu_x^{\infty}(A)\,d\lambda_{\nu}(x),
\end{align*}
where $f\in \mathbf{E}(\Omega;\R^{l})$ and $\nu\in\mathbf{Y}(\Omega,\mu;\R^{l})$, the space of Young measures can be considered a part of the dual space $\mathbf{E}(\Omega;\R^{l})^*$. We say that a sequence of Young measures $(\nu_j)\subset \mathbf{Y}(\Omega,\mu;\R^{l})$ converges weakly* to $\nu\in\mathbf{Y}(\Omega,\mu;\R^{l})$ if $\langle\!\langle f,\nu_j\rangle\!\rangle\to \langle\!\langle f,\nu\rangle\!\rangle$ for every $f\in \mathbf{E}(\Omega;\R^{l})$.

To every Radon measure $\gamma\in \mathcal M(\Omega;\R^{l})$, with Radon-Nikodym decomposition with respect to $\mu$ written as $\gamma=\frac{d\gamma}{d\mu}\,\mu+\gamma^{s,\mu}$, we associate an \emph{elementary Young measure} $\eps_{\gamma}\in\mathbf{Y}(\Omega,\mu;\R^{l})$ by
\[
(\eps_{\gamma})_x:=\delta_{\frac{d\gamma}{d\mu}(x)},\quad \lambda_{\eps_{\gamma}}:=|\gamma^{s,\mu}|,\quad (\eps_{\gamma})_x^{\infty}:=\delta_{p(x)},
\]
where $p:=\frac{d\gamma^{s,\mu}}{d|\gamma^{s,\mu}|}\in L^1(\Omega,|\gamma^{s,\mu}|;\partial\mathbb{B}^{l})$.

Crucially, we have the following.

\begin{theorem}\label{thm:generation of generalized Young measures for general mu}
Let $\mu\in\mathcal M^+(\overline{\Omega})$ with $\mu(\partial\Omega)=0$, and let $(\gamma_j)\subset \mathcal M(\overline{\Omega};\R^l)$ be a sequence of Radon measures that is bounded in the total variation norm, that is, $\sup_{j\in\N}|\gamma_j|(\overline{\Omega})<\infty$. Then there exists a subsequence (not relabeled) and a generalized Young measure $(\nu_x,\lambda_{\nu},\nu_x^{\infty})$ with
\begin{alignat}{2}
\nonumber &(\nu_x)_x \in L_{w*}^{\infty}(\Omega,\mu;\mathcal M^1(\R^l)),\qquad\qquad
&&\lambda_{\nu}\in\mathcal M^+(\overline{\Omega}),\\
\label{eq:property of Young measure wrt mu} &(\nu_x^{\infty})_x \in L_{w*}^{\infty}(\overline{\Omega},\lambda_{\nu};\mathcal M^1(\partial\mathbb B^l)),
&&x\mapsto\langle |\cdot|,\nu_x \rangle\in \L^1(\Omega,\mu),
\end{alignat}
such that $\langle\!\langle f,\eps_{\gamma_j}\rangle\!\rangle\to \langle\!\langle f,\nu\rangle\!\rangle$, or equivalently
\begin{equation}\label{eq:generation of generalized Young measures for general mu}
\begin{split}
& f\left(x,\frac{d\gamma_j}{d\mu}(x)\right)\,\mu+ f^{\infty}\left(x,\frac{d\gamma_j^{s,\mu}}{d|\gamma_j^{s,\mu}|}(x)\right)|\gamma_j^{s,\mu}|\\
&\qquad\qquad \overset{*}{\rightharpoondown} \langle f(x,\cdot),\nu_x\rangle \,\mu +\langle f^{\infty}(x,\cdot),\nu_x^{\infty}\rangle\,\lambda_{\nu}\quad \textrm{in }\mathcal M(\overline{\Omega})
\end{split}
\end{equation}
for every $f\in \mathbf{E}(\Omega;\R^{N\times n})$.
\end{theorem}

\begin{proof}
This is proved in the case $\mu=\mathcal L^n\mres \Omega$ in \cite[Lemma 2, Corollary 2, Theorem 7]{KR2}, but the proofs run through also if we replace the Lebesgue measure by a more general $\mu$.

See also \cite[Theorem 2.5]{AB} for a proof in the case $\gamma_j^s\equiv 0$.
\end{proof}

\begin{corollary}\label{cor:representation wrt mu}
In the above theorem, \eqref{eq:generation of generalized Young measures for general mu} holds also for every $\mu\times\mathcal{B}(\R^l)$-me\-asu\-ra\-ble $f\in\mathbf{R}(\Omega;\R^l)$.

In the case $\mu=\mathcal L^n$, \eqref{eq:generation of generalized Young measures for general mu} also holds for every Carath\'eodory integrand $f\colon\overline{\Omega}\times \R^{l}\to \R$ possessing a recession function $f^{\infty}\colon\overline{\Omega}\times \R^l\to \R$ in the sense of \eqref{eq:definition of recession function} for $(x,A)\in (\overline{\Omega}\setminus N)\times \R^l$, and $f^{\infty}$ is jointly continuous in $(\overline{\Omega}\setminus N)\times \R^l$, where $N\subset \overline{\Omega}$ is a Borel set with $(\mathcal L^n+\lambda_{\nu})(N)=0$.
\end{corollary}

Note that a Carath\'eodory function $f\colon\Omega\times \R^N\to \R$ is by definition $\mathcal L^n\times \mathcal B(\R^N)$-measurable, but here we need the assumption of $\mu\times\mathcal{B}(\R^N)$-measurability.

\begin{proof}
Again, this is proved in the case $\mu=\mathcal L^n\mres \Omega$ in \cite[Proposition 2]{KR2}, but the proof runs through also in the general case with the assumption of $\mu\times\mathcal{B}(\R^N)$-measurability.
\end{proof}

In particular, given $u\in\BV(\Omega;\R^N)$, we can associate to its derivative $Du\in\mathcal M(\Omega;\R^{N\times n})$ the Radon-Nikodym decomposition $Du=\frac{dDu}{d\mu}\,\mu+D^{s,\mu}u$, and then the elementary Young measure $\eps_{Du}\in \mathbf{Y}(\Omega,\mu;\R^{N\times n})$ with
\[
(\eps_{Du})_x:=\delta_{\frac{dDu}{d\mu}},\quad \lambda_{\eps_{Du}}:=|D^{s,\mu} u|,\quad (\eps_{Du})_x^{\infty}:=\delta_{p(x)},
\]
where $p:=\frac{D^{s,\mu} u}{|D^{s,\mu} u|}\in L^1(\Omega,|D^{s,\mu} u|;\partial \mathbb{B}^{N\times n})$.

For a norm-bounded sequence $(u_j)\subset \BV(\Omega;\R^N)$,
we say that the derivatives $Du_j$ \emph{generate} the generalized Young measure
\[
\nu=(\nu_x,\lambda_\nu,\nu_x^{\infty})\in \mathbf{Y}(\Omega;\R^{N\times n}),
\]
if for all $f\in \mathbf{E}(\Omega,\R^{N\times n})$ we have that $\langle\!\langle f,\eps_{Du_j}\rangle\!\rangle\to \langle\!\langle f,\nu\rangle\!\rangle$ for all $f\in \mathbf{E}(\Omega;\R^{N\times n})$, or equivalently
\begin{equation}\label{eq:generation of gradient Young measure wrt mu}
\begin{split}
& f\left(x,\frac{dDu_j}{d\mu}\right)\,\mu+ f^{\infty}\left(x,\frac{dD^{s,\mu} u_j}{d|D^{s,\mu} u_j|}(x)\right)|D^{s,\mu} u_j|\\
&\qquad\quad \overset{*}{\rightharpoondown} \langle f(x,\cdot),\nu_x\rangle \,\mu+\langle f^{\infty}(x,\cdot),\nu_x^{\infty}\rangle\,\lambda_{\nu}\qquad \textrm{in }\mathcal M(\overline{\Omega}).
\end{split}
\end{equation}
We call such a generalized Young measure a \emph{gradient Young measure}. Since $(u_j)$ is norm-bounded, we have $u_j\overset{*}{\rightharpoondown}u$ for some $u\in\BV(\Omega;\R^N)$.
The \emph{barycenter} of a generalized Young measure $\nu\in\mathbf{Y}(\Omega;\R^{N\times n})$ is defined as the measure
\[
\langle \id,\nu_x\rangle\, \mu+\langle\id,\nu_x^{\infty}\rangle\lambda_{\nu}.
\]
Note that by choosing $f$ to be the identity on $\R^{N\times n}$ in \eqref{eq:generation of gradient Young measure wrt mu} (componentwise, to be precise), we obtain that $Du$ is the restriction of the barycenter to $\Omega$.

In the case $\mu=\mathcal L^n$, we have the following Jensen's inequalities for gradient Young measures, which are part of \cite[Theorem 9]{KR2}.

\begin{theorem}\label{thm:Jensens inequalities for generalized Young measures}
Let $u\in\BV(\Omega;\R^N)$ and let $\nu\in\mathbf{Y}(\Omega;\R^{N\times n})$, $\nu=(\nu_x,\lambda_{\nu},\nu_x^{\infty})$ be a gradient Young measure with barycenter $Du$ and satisfying $\lambda(\partial\Omega)=0$. Then the following hold for any quasiconvex $f\colon\R^{N\times n}\to \R$ with linear growth (that is, $|F(A)|\le M(|A|+1)$ for all $A\in \R^{N\times n}$ and some $M\ge 0$):
\begin{align*}
& f(\nabla u(x))\le \langle f,\nu_x\rangle+\langle f^{\#},\nu_x^{\infty}\rangle\frac{d\lambda_{\nu}}{d\mathcal L^n}(x)\quad\textrm{for }\mathcal L^n\textrm{-almost every } x\in\Omega,\\
& f^{\#}\left(\frac{dD^s u}{d|D^s u|}\right)|D^s u|\le \langle f^{\#},\nu_x^{\infty}\rangle\lambda^s_{\nu}\quad\textrm{as measures}.
\end{align*}

\end{theorem}


\section{The integral representation}
Let $F\colon\R^{N\times n}\mapsto \R$ be quasiconvex, with linear growth
\[
m|A|\le F(A)\le M(1+|A|)\qquad\text{for all }A\in \R^{N\times n},
\]
for some $0<m\le M$, such that the recession function $F^{\infty}$ exists in the sense of \eqref{eq:definition of recession function}. Let $\Omega\subset \R^{n}$ be a bounded open set with $\mathcal L^n(\partial\Omega)=0$, and let $\mu\in\mathcal M^+(\Omega)$ with $\mathcal L^n\ll\mu$. We define a Sobolev space with respect to $\mu$ by
\[
W^{1,1}_{\mu}(\Omega;\R^{N}):=\{u\in \BV(\Omega;\R^N):\,Du\ll \mu\}.
\]
We consider the functional
\begin{equation}\label{eq:definition by relaxation}
\begin{split}
\mathcal F_{*}(u,\Omega)
\ :=\inf\Big\{ \liminf_{j\to \infty} \int_{\Omega} F\left(\frac{d Du_j}{d\mu}\right) d\mu,\quad &u_j\in W^{1,1}_{\mu}(\Omega;\R^{N})\\
&u_j\to u\text{ in } L^{1}(\Omega;\R^N) \Big\}
\end{split}
\end{equation}
for $u\in\BV(\Omega;\R^N)$.
Note that the convergence above is in $L^{1}(\Omega;\R^N)$ with respect to the Lebesgue measure $\mathcal L^{n}$, not $\mu$. We will prove an integral representation for the above functional. The representation is
\begin{equation}\label{eq:integral representation}
\int_{\Omega}F\left(\frac{dDu}{d\mu}\right)\,d\mu+\int_{\Omega}F^{\infty}\left(\frac{dD^{s,\mu}u}{d|D^{s,\mu}u|}\right)\,d|D^{s,\mu}u|
\end{equation}
for any $u\in \BV(\Omega;\R^N)$, where $D^{s,\mu}u$ is the singular part of the variation measure $Du$ with respect to $\mu$.

Initially we will work with a more restricted class of integrands, defined as follows.

\begin{definition}\label{def:special quasiconvex integrands}
Define the class $\mathbf{SQ}(\R^{N\times n})$ of \emph{special quasiconvex integrands} as quasiconvex functions $F\colon\R^{N\times n}\to \R$ with linear growth $|F(A)|\le M(1+|A|)$ for some $M\ge 0$, such that for some
parameters $i,r_i>0$, $F(A)=F^{\infty}(A)-i$ for $|A|\ge r_i$, and $F^{\infty}(A)\ge |A|/i$ for all $A\in \R^{N\times n}$.
\end{definition}

Note that the existence of the recession function $F^{\infty}$ in the sense of \eqref{eq:definition of recession function} is part of the definition. (We could equally well require above that $F(A)=F^{\#}(A)-i$ for $|A|\ge r_i$, recall \eqref{eq:definition of generalized recession function}, as this would imply the existence of $F^{\infty}$.)
Clearly $\mathbf{SQ}(\R^{N\times n})\subset \mathbf{E}(\Omega;\R^{N\times n})$ (constant in the $x$-variable).

Given $F\in\mathbf{Q}(\R^{N\times n})$ with linear growth $0\le F(A)\le M(1+|A|)$ for some $M\ge 0$, we can define $G_i(A):=\max\{F(A),\,F^{\#}(A)+|A|/i-i\}$ for each $i\in\N$, and then it is shown in \cite[Lemma 6.3]{KK} that $G_i\in\mathbf{SQ}(\R^{N\times n})$ and that $G_i(A)\searrow F(A)$ and $(G_i)^{\infty}(A)\searrow F^{\infty}(A)$ for every $A\in \R^{N\times n}$. We will use this fact on a number of occasions.

\subsection{Estimate from below}

In order to obtain the integral representation, we first prove the estimate from below. 

\begin{proposition}\label{prop:estimate from below}
Let $\Omega\subset \R^{n}$ be a bounded open set with $\mathcal L^n(\partial\Omega)=0$, let $\mu\in\mathcal M^+(\Omega)$ with $\mathcal L^n\ll \mu$, let $F\in \mathbf{R}(\Omega;\R^{N\times n})\cap \mathbf{Q}(\R^{N\times n})$ with
\[
m|A|\le F(A)\le M(1+|A|)
\]
for some $0<m\le M$,
and let $u\in\BV(\Omega;\R^N)$. Then we have
\[
\mathcal F_{*}(u,\Omega)\ge \int_{\Omega}F\left(\frac{dDu}{d\mu}\right)\,d\mu+\int_{\Omega}F^{\infty}\left(\frac{dD^{s,\mu}u}{d|D^{s,\mu}u|}\right)\,d|D^{s,\mu}u|.
\]
\end{proposition}

Write the Radon-Nikodym decomposition of $\mu$ as $\mu=a\,\mathcal L^{n}+\mu^{s}$, with $a\in L^1(\Omega)$. We prove the theorem by considering separately the sets where the absolutely continuous part and the singular part of $\mu$ are carried.

\subsubsection{The absolutely continuous part}

The following lemma gives, in essence, the estimate from below for the set where the absolutely continuous part of $\mu$ is carried. At this point, we make the extra assumption that $F$ is a special quasiconvex integrand.

\begin{lemma}\label{lem:estimate from below}
Let $\Omega\subset \R^n$ be a bounded open set with $\mathcal L^n(\partial\Omega)=0$, let $\mu\in\mathcal M^+(\Omega)$ with $\mathcal L^n\ll \mu$, and let $F\in \mathbf{SQ}(\R^{N\times n})$ with parameters $i,r_i>0$, and with linear growth $0\le F(A)\le M(1+|A|)$.
Then for any open $U\subset\Omega$ and any sequence $(u_j)\subset W^{1,1}_{\mu}(\Omega;\R^{N})$ with $u_j\to u$ in $L^1(\Omega;\R^N)$ and
\begin{equation}\label{eq:finiteness of sequence in lower semicontinuity}
\liminf_{j\to\infty}\int_{\Omega}F\left(\frac{dDu_j}{d\mu}\right)\,d\mu<\infty,
\end{equation}
we have
\begin{equation}\label{eq:estimate from below where mua lives}
\begin{split}
&\liminf_{j\to\infty} \int_{U} F\left(\frac{dDu_j}{d\mu}\right)\,d\mu\\
&\qquad\ge \int_{U} F\left(\frac{\nabla u}{a}\right) a\,d\mathcal L^{n}+\int_{U}F^{\infty}\left(\frac{dD^s u}{d|D^s u|}\right)\,d|D^{s}u|-(Mr_i+i)\mu^s(U).
\end{split}
\end{equation}
\end{lemma}


\begin{proof}
Since $\mathcal L^n\ll \mu$, we can assume that $a>0$ everywhere in $\Omega$.
Pick a subsequence of $(u_j)$ (not relabeled) that gives the limit in \eqref{eq:estimate from below where mua lives}.
Since $F(A)=F^{\infty}(A)-i\ge |A|/i-i$ for all $|A|\ge r_i$, we have by \eqref{eq:finiteness of sequence in lower semicontinuity} that $(u_j)$ is a norm-bounded sequence in $\BV(\Omega;\R^{N})$.
Thus $u_j\overset{*}{\rightharpoondown} u$ in $\BV(\Omega;\R^{N})$.
By Theorem \ref{thm:generation of generalized Young measures for general mu}, the derivatives $Du_j$ generate a generalized Young measure $(\nu_x,\lambda_{\nu},\nu_x^{\infty})$ with respect to the Lebesgue measure, with $\lambda_{\nu}\in \mathcal M^{+}(\overline{\Omega})$ and
\[
(\nu_x)_x\in L^{\infty}_{w^*}(\Omega;\mathcal M^1(\R^{N\times n})),\qquad(\nu_x^{\infty})_x\in L^{\infty}_{w^*}(\overline{\Omega},\lambda_{\nu};\mathcal M^1(\partial \mathbb{B}^{N\times n})).
\]
This means that for every representation integrand $f\in \mathbf{R}(\Omega;\R^{N\times n})$ and every integrand satisfying the conditions of the latter part of Corollary \ref{cor:representation wrt mu}, we have
\begin{equation}\label{eq:generation of gradient Young measure in U}
\begin{split}
&f (x,\nabla u_j(x))\,\mathcal L^{n}\mres \Omega+f^{\infty}\left(x,\frac{dD^{s}u_j}{d|D^{s}u_j|}\right)\,|D^{s}u_j| \\
&\qquad\qquad\overset{*}{\rightharpoondown} \langle f(x,\cdot),\nu_x\rangle\,\mathcal L^{n}\mres \Omega+\langle f^{\infty}(x,\cdot),\nu_x^{\infty}\rangle \,\lambda_{\nu}\qquad\textrm{in }\mathcal M(\overline{\Omega}).
\end{split}
\end{equation}
First assume that $\mathcal L^n(\partial U)=\lambda_{\nu}(\partial U)=0$. Let us start computing
\begin{equation}\label{eq:division into absolutely continuous and singular part}
\begin{split}
\int_{U} F\left(\frac{dDu_j}{d\mu}\right)d\mu &= \int_{U} F\left(\frac{dDu_j}{d (a\mathcal L^{n})}\right)a\, d\mathcal L^{n} + \int_{U} F\left(\frac{dDu_j}{d\mu^{s}}\right)d\mu^{s}\\
&=\int_{U} F\left( \frac{\nabla u_j}{a}\right)a\,d\mathcal L^{n}+\int_{U}F\left(\frac{dD^s u_j}{d\mu^{s}}\right)d\mu^{s}\\
&=:I_j+II_j.
\end{split}
\end{equation}
We wish to analyze the term $I_j$ by using the fact that $Du_j$ generates a generalized Young measure. However, the function
\[
(x,A)\mapsto F\left(\frac{A}{a(x)}\right)a(x)\mathbbm{1}_{U}(x)
\]
does not necessarily satisfy the conditions of the latter part of Corollary \ref{cor:representation wrt mu}: while it is a Cara\-th\'eo\-dory function, its recession function need not be continuous as required. To overcome this problem, we define the super-level sets of $a$:
\[
E_m:=\{x\in \Omega:\, a(x)> m\},\qquad  m\in \N.
\]
Recall that $a(x)>0$ for every $x\in\Omega$. Denoting the minimum of $a$ and $m$ by $a\wedge m$, by the fact that $F(A)=F^{\infty}(A)-i$ for all $|A|\ge r_i$ we have for any $x\in U$ and $A\in\partial\mathbb{B}^{N\times n}$
\begin{align*}
&\limsup_{\substack{x'\to x \\ A'\to A \\ t\to\infty}} \frac{F\left(\frac{tA'}{a(x')\wedge m}
\right)}{t}
 a(x')\wedge m
= \limsup_{\substack{A'\to A\\ t\to\infty}}
\frac{F(tA')}{t}\\
&\qquad=\limsup_{\substack{A'\to A\\ t\to\infty}}\frac{F^{\infty}(tA')-i}{t}
=\limsup_{\substack{A'\to A\\ t\to\infty}}\frac{tF^{\infty}(A')}{t}= 
\limsup_{A'\to A}F^{\infty}(A')=F^{\infty}(A)
\end{align*}
by the (Lipschitz) continuity of $F^{\infty}$.
Note that the first equality is not necessarily true unless we take the minimum of $a$ with $m$.
Also, we now see that all of the limit superiors above are in fact limits. We conclude that
\begin{equation}\label{eq:modified representation integrand}
(x,A)\mapsto F\left(\frac{A}{(a\wedge m)(x)}\right)(a\wedge m)(x)\mathbbm{1}_{U}(x)
\end{equation}
satisfies the conditions of the latter part of Corollary \ref{cor:representation wrt mu}. Fix $m\in\N$. By the fact that $F(A)=F^{\infty}(A)-i$ for all $|A|\ge r_i$, we can write
\begin{align*}
& I_j - \int_{U} F\left(\frac{\nabla u_j}{a\wedge m}\right)a\wedge m\,d\mathcal L^{n}=\int_{U\cap E_m}F\left(\frac{\nabla u_j}{a}\right)a-F\left(\frac{\nabla u_j}{m}\right)m\,d\mathcal L^{n}\\
  &\qquad= \int_{U\cap E_m\cap \{|\nabla u_j|<ar_i\}} F\left(\frac{\nabla u_j}{a}\right)a-F\left(\frac{\nabla u_j}{m}\right)m\,d\mathcal L^{n}\\
 &\qquad\qquad+\int_{U\cap E_m\cap \{|\nabla u_j|\ge ar_i\}} F^{\infty}\left(\frac{\nabla u_j}{a}\right)a-F^{\infty}\left(\frac{\nabla u_j}{m}\right)m-i(a-m)\,d\mathcal L^{n}\\
   &\qquad=\int_{U\cap E_m\cap \{|\nabla u_j|<ar_i\}} F\left(\frac{\nabla u_j}{a}\right)a-F\left(\frac{\nabla u_j}{m}\right)m\,d\mathcal L^{n}\\
   &\qquad\qquad -\int_{U\cap E_m\cap \{|\nabla u_j|\ge ar_i\}}i(a-m)\,d\mathcal L^{n}\\
   &\qquad :=\eps_m.
\end{align*}
We have by the linear growth of $F$
\begin{align*}
&\bigg| \int_{U\cap E_m\cap \{|\nabla u_j|<ar_i\}} F \left(\frac{\nabla u_j}{a}\right)a- F\left(\frac{\nabla u_j}{m}\right)m\,d\mathcal L^{n}\bigg|\\
&\qquad\qquad\le \int_{U\cap E_m\cap \{|\nabla u_j|<ar_i\}} F \left(\frac{\nabla u_j}{a}\right)a+ F\left(\frac{\nabla u_j}{m}\right)m\,d\mathcal L^{n} \\
&\qquad\qquad\le \int_{U\cap E_m} M(1+r_i)a+M(1+ar_i/m)m\,d\mu\\
&\qquad\qquad\le \int_{U\cap E_m} 2Ma(1+r_i)\,d\mu.
\end{align*}
Clearly this last quantity converges to zero as $m\to\infty$, as does the second term of $\eps_m$, so in total $\eps_m\to 0$ as $m\to \infty$.

By the fact that the derivatives $Du_j$ generate a generalized Young measure (recall \eqref{eq:generation of gradient Young measure in U}) and the fact that
the integrand
\eqref{eq:modified representation integrand} satisfies the conditions of the latter part of Corollary \ref{cor:representation wrt mu} and has recession function $F^{\infty}$ in $U$, we have
\begin{equation}\label{eq:estimate of absolutely continuous term}
\begin{split}
I_j & +\int_{U} F^{\infty}\left(\frac{dD^{s}u_j}{d|D^{s}u_j|}\right)\,d|D^{s}u_j|-\eps_m \\
&=\int_{U} F\left(\frac{\nabla u_j}{a\wedge m}\right)a\wedge m\,d\mathcal L^{n}+\int_{U}F^{\infty}\left(\frac{dD^{s}u_j}{d|D^{s}u_j|}\right)\,d|D^{s}u_j|\\
& \to \int_{U} \int_{\R^{N\times n}}F\left(\frac{A}{a\wedge m}\right) a\wedge m\,d\nu_x(A)\,d\mathcal L^{n}+\int_{U}\int_{\partial \mathbb B^{N\times n}} F^{\infty}(A)\, d\nu_x^{\infty}(A)\,d\lambda_{\nu}
\end{split}
\end{equation}
as $j\to\infty$.
Recalling \eqref{eq:division into absolutely continuous and singular part}, let us then consider the term $II_j$.
Since $F(A)=F^{\infty}(A)-i$ for all $|A|\ge r_i$, we estimate
\begin{equation}\label{eq:estimate of singular term}
\begin{split}
&II_j- \int_{U}  F^{\infty}\left(\frac{dD^{s}u_j}{d|D^{s}u_j|}\right)\,d|D^{s}u_j|\\
&\ \   = \int_{U}F\left(\frac{dD^s u_j}{d\mu^{s}}\right)\,d\mu^{s}- \int_{U} F^{\infty}\left(\frac{dD^{s}u_j}{d\mu^s}\right)\,d\mu^s\\
&\ \  = \int_{U\cap \{|dD^s u_j/d\mu^{s}|<r_i\}}  F\left(\frac{dD^{s}u_j}{d\mu^s}\right)  \,d\mu^s-i\mu^s(U\cap \{|dD^s u_j/d\mu^{s}|\ge r_i\})\\
&\ \ \qquad-\int_{U\cap \{|dD^s u_j/d\mu^{s}|<r_i\}} F^{\infty}\left(\frac{dD^{s}u_j}{d\mu^s}\right)\,d\mu^s\\
&\ \   \ge -i\mu^s(U\cap \{|dD^s u_j/d\mu^{s}|\ge r_i\})-\int_{U\cap \{|dD^s u_j/d\mu^{s}|<r_i\}} F^{\infty}\left(\frac{dD^{s}u_j}{d\mu^s}\right)\,d\mu^s\\
&\ \   \ge -i\mu^s(U)-\int_{U\cap \{|dD^s u_j/d\mu^{s}|<r_i\}} M\left|\frac{dD^{s}u_j}{d\mu^s}\right|\,d\mu^s\\
&\ \  \ge-i\mu^s(U)-Mr_i\mu^s(U).
\end{split}
\end{equation}
Combining \eqref{eq:division into absolutely continuous and singular part}, \eqref{eq:estimate of absolutely continuous term}, and \eqref{eq:estimate of singular term}, we get by Jensen's inequalities for generalized Young measures given in Theorem \ref{thm:Jensens inequalities for generalized Young measures},
\begin{align*}
&\liminf_{j\to\infty}  \int_{U} F\left(\frac{dDu_j}{d\mu}\right)\,d\mu =\liminf_{j\to\infty}(I_j+II_j) \\
 &\quad\ge \int_{U} \int_{\R^{N\times n}}F\left(\frac{A}{a\wedge m}\right) a\wedge m\,d\nu_x(A)\,d\mathcal L^{n}+\int_{U}\int_{\partial \mathbb B^{N\times n}} F^{\infty}(A)\, d\nu_x^{\infty}(A)\,d\lambda_{\nu}\\
 &\qquad\quad+\eps_m-(Mr_i+i)\mu^s(U)\\
&\quad\ge \int_{U} F\left(\frac{\nabla u}{a\wedge m}\right) a\wedge m\,d\mathcal L^{n}+\int_{U}F^{\infty}\left(\frac{dD^s u}{d|D^s u|}\right)\,d|D^{s}u|\\
&\qquad\quad+\eps_m-(Mr_i+i)\mu^s(U)\\
&\quad\ge \int_{U} F\left(\frac{\nabla u}{a}\right) a\wedge m\,d\mathcal L^{n}+\int_{U}F^{\infty}\left(\frac{dD^s u}{d|D^s u|}\right)\,d|D^{s}u|\\
&\qquad\quad+\eps_m-(Mr_i+i)\mu^s(U)\\
&\quad\to \int_{U} F\left(\frac{\nabla u}{a}\right) a\,d\mathcal L^{n}+\int_{U}F^{\infty}\left(\frac{dD^s u}{d|D^s u|}\right)\,d|D^{s}u|-(Mr_i+i)\mu^s(U)
\end{align*}
as $m\to\infty$, by the monotone convergence theorem.
Finally, if $U$ does not satisfy $\lambda_{\nu}(\partial U)=0$ or $\mathcal L^n(\partial U)=0$,
we define
\[
U_{\kappa}:=\{x\in U:\,\dist(x,U^{c})>\kappa\},\qquad \kappa>0,
\]
and then $\lambda_{\nu}(\partial U_{\kappa})=0$ and $\mathcal L^n(\partial U_{\kappa})=0$ for all but at most countably many $\kappa>0$ by the fact that these are finite measures on $U$.
For such values of $\kappa$ we write
\begin{align*}
&\liminf_{j\to\infty} \int_{U} F\left(\frac{dDu_j}{d\mu}\right)\,d\mu\\
&\quad \ge \liminf_{j\to\infty}\int_{U_{\kappa}} F\left(\frac{dDu_j}{d\mu}\right)\,d\mu\\
&\quad \ge \int_{U_{\kappa}} F\left(\frac{\nabla u}{a}\right) a\,d\mathcal L^{n}+\int_{U_{\kappa}}F^{\infty}\left(\frac{dD^s u}{d|D^s u|}\right)\,d|D^{s}u|-(Mr_i+i)\mu^s(U_{\kappa})\\
&\qquad \to \int_{U} F\left(\frac{\nabla u}{a}\right) a\,d\mathcal L^{n}+\int_{U}F^{\infty}\left(\frac{dD^s u}{d|D^s u|}\right)\,d|D^{s}u|-(Mr_i+i)\mu^s(U)
\end{align*}
as $\kappa\to 0$, by the monotone convergence theorem.
\end{proof}

\subsubsection{The singular part}

Let us then consider the set where $\mu^{s}$ is carried. We prove the following lemma.
\begin{lemma}\label{lem:estimate from below for singular part of mu}
Let $\Omega\subset \R^n$ be a bounded open set, let $\mu\in\mathcal M^+(\Omega)$, and let $F\in \mathbf{SQ}(\R^{N\times n})$ with $F\ge 0$.
Then for any sequence $(u_j)\subset W^{1,1}_{\mu}(\Omega;\R^{N})$ with $u_j\to u$ in $L^1(\Omega;\R^N)$ and
\[
\liminf_{j\to\infty}\int_{\Omega}F\left(\frac{dDu_j}{d\mu}\right)\,d\mu<\infty,
\]
we have for any ball $B(y,r)\subset\Omega$
\begin{equation}\label{eq:lower semicontinuity for singular part of mu}
\int_{B(y,r)}F\left(\frac{dD^{s}u}{d\mu^{s}}\right)\,d\mu^{s}\le \liminf_{j\to\infty}\int_{B(y,r)}F\left(\frac{dDu_j}{d\mu}\right)\,d\mu.
\end{equation}
\end{lemma}

\begin{proof}
Note again that it is enough to prove the result for a subsequence.
Let $i,r_i>0$ be the parameters of $F$, see Definition \ref{def:special quasiconvex integrands}.
Since $F(A)=F^{\infty}(A)-i\ge |A|/i-i$ for all $|A|\ge r_i$, the sequence $\frac{dDu_j}{d\mu}$ is norm-bounded in $L^1(\Omega,\mu;\R^{N\times n})$,
implying that $(u_j)$ is a norm-bounded sequence in $\BV(\Omega;\R^N)$.
By Theorem \ref{thm:generation of generalized Young measures for general mu} we know that with respect to $\mu$, a subsequence of $Du_j$ (not relabeled) generates a generalized Young measure $(\nu_x,\lambda_{\nu},\nu_x^{\infty})$, with $\lambda_{\nu}\in \mathcal M^{+}(\overline{\Omega})$ and
\[
(\nu_x)_x\in L^{\infty}_{w^*}(\Omega,\mu;\mathcal M^1(\R^{N\times n})),\qquad (\nu_x^{\infty})_x\in L^{\infty}_{w^*}(\overline{\Omega},\mu;\mathcal M^1(\partial\mathbb{B}^{N\times n})).
\]
This means in particular that for every integrand $f\in \mathbf{E}(\Omega;\R^{N\times n})$,
\begin{equation}\label{eq:Young measure wrt to mu}
\begin{split}
f\left(\frac{dDu_j}{d\mu}\right)&\mu \overset{*}{\rightharpoondown} \langle f,\nu_x\rangle \mu+\langle f^{\infty},\nu_x^{\infty}\rangle\lambda_{\nu}\\
&= \left(\langle f,\nu_x\rangle a+\langle f^{\infty},\nu_x^{\infty}\rangle \frac{d\lambda_{\nu}}{d\mathcal L^n}\right)\mathcal L^n+\langle f,\nu_x\rangle \mu^{s}+\langle f^{\infty},\nu_x^{\infty}\rangle\lambda_{\nu}^{s}
\end{split}
\end{equation}
in $\mathcal M(\Omega)$.
By Alberti's rank one theorem, see \cite{Alb}, we have $\mu^{s}$-almost everywhere that
\begin{equation}\label{eq:rectifiability of singular part of mu}
\frac{dDu_j}{d\mu^{s}}= \xi_j\otimes \eta,\quad \qquad \frac{dDu}{d\mu^{s}}= \xi\otimes \eta,
\end{equation}
where $\xi_j,\xi\in \R^{N}$ and $\eta\in \partial \mathbb B^{n}$. Note that $\eta$ does not depend on $j$.
We show that for $\mu^{s}$-almost every $x\in \Omega$, the measure $\nu_x$ is carried on the hyperplane $\R^{N}\otimes\eta(x)$. For this, fix $\eps>0$ and fix a point $x_0\in \Omega$. Excluding a $\mu^s$-negligible set, we can assume by the Besicovitch differentiation theorem (see e.g. \cite[Theorem 2.22]{AFP}) that for some radius $r>0$, we have $B(x_0,r)\subset \Omega$ and
\begin{equation}\label{eq:conditions for point in singular set of mu}
\vint{B(x_0,r)}|\eta-\eta(x_0)|\,d\mu^{s}<\eps\quad\ \textrm{and}\quad\ \int_{B(x_0,r)}a\,d\mathcal L^n<\eps\mu(B(x_0,r)).
\end{equation}
Fix $R\ge 1$ and define
\[
f(A):=\min\{1,\dist (A,(\R^{N}\otimes \eta(x_0))\cup B(0,R)^c)\};
\]
note that there is no $x$-dependence, and $f\in \mathbf{E}(\Omega;\R^{N\times n})$.
Since $f(\xi\otimes\eta)=0$ for $|\xi|\ge R$ and $|\eta|=1$ and since $f$ is $1$-Lipschitz,
\begin{equation}\label{eq:consequence of Lebesgue point for eta}
\begin{split}
&\left|\ \vint{B(x_0,r)}f \left(\xi_j\otimes \eta\right)d\mu^{s}-\vint{B(x_0,r)}f \left(\xi_j\otimes \eta(x_0)\right)d\mu^{s}\,\right|\\
&\qquad \le R\vint{B(x_0,r)}|\eta-\eta(x_0)|\,d\mu^s<R\eps
\end{split}
\end{equation}
by \eqref{eq:conditions for point in singular set of mu}.
Since $\langle f,\nu_x\rangle\in L^1(\Omega,\mu)$ by \eqref{eq:property of Young measure wrt mu}, excluding a further $\mu^s$-negligible set and possibly making $r>0$ smaller, we can also assume
that
\begin{equation}\label{eq:lebesgue point for nu_x}
\vint{B(x_0,r)}|\langle f, \nu_x\rangle -\langle f, \nu_{x_0}\rangle|\,d\mu<\eps.
\end{equation}
Clearly $f^{\infty}\equiv 0$ and then by \eqref{eq:Young measure wrt to mu}, we have
\[
f\left(\frac{dDu_j}{d\mu}\right)\mu \overset{*}{\rightharpoondown} \langle f,\nu_x\rangle a\,\mathcal L^n+\langle f,\nu_x\rangle \,\mu^{s}\qquad\textrm{in }\mathcal M(\Omega).
\]
Now by \eqref{eq:lebesgue point for nu_x},
\begin{alignat}{2}
\nonumber \langle f,\nu_{x_0}\rangle &\le \vint{B(x_0,r)}\langle f,\nu_x\rangle \,d\mu+\eps\\
\nonumber & \le \liminf_{j\to\infty}\,\vint{B(x_0,r)} f \left(\frac{dDu_j}{d\mu}\right)d\mu+\eps\\
\nonumber & \overset{f\le 1}{\le} \liminf_{j\to\infty}\,\vint{B(x_0,r)}f \left(\frac{dD u_j}{d\mu^s}\right)d\mu^s+\frac{\int_{B(x_0,r)}a\,d\mathcal L^n}{\mu(B(x_0,r))}+\eps\\
\nonumber &\overset{\eqref{eq:conditions for point in singular set of mu}}{\le} \liminf_{j\to\infty}\,\vint{B(x_0,r)}f \left(\frac{dD u_j}{d\mu^s}\right)d\mu^{s}+2\eps\\
\nonumber &\overset{\eqref{eq:rectifiability of singular part of mu}}{=} \liminf_{j\to\infty}\,\vint{B(x_0,r)}f \left(\xi_j\otimes \eta\right)d\mu^{s}+2\eps\\
\nonumber &\overset{\eqref{eq:consequence of Lebesgue point for eta}}{\le} \liminf_{j\to\infty}\,\vint{B(x_0,r)}f \left(\xi_j\otimes \eta(x_0)\right)d\mu^{s}+R\eps+2\eps\\
\nonumber &\le 3R\eps,
\end{alignat}
since $f$ is zero on the hyperplane $\R^N\otimes \eta(x_0)$.
Letting $\eps\to 0$, we get $\langle f,\nu_{x_0}\rangle=0$, implying that $\nu_{x_0}$ is carried on the set $(\R^{N}\otimes \eta(x_0))\cup B(0,R)^{c}$. Letting $R\to \infty$, we obtain that $\nu_{x_0}$ is carried on the hyperplane $\R^{N}\otimes \eta(x_0)$.

By choosing $f$ to be the identity mapping on $\R^{N\times n}$ in \eqref{eq:Young measure wrt to mu} (componentwise, to be precise), and noting that $Du_j\overset{*}{\rightharpoondown} Du$ in $\mathcal M(\Omega)$ (the fact that $(u_j)$ is a norm-bounded sequence in $\BV(\Omega;\R^N)$ implies that $u_j\overset{*}{\rightharpoondown} u$ in $\BV(\Omega;\R^N)$), we get for the singular parts 
\begin{equation}\label{eq:singular part of variation measure and Young measure wrt mu}
D^{s}u=\langle \id,\nu_x\rangle\,\mu^{s}+\langle \id,\nu_x^{\infty}\rangle\,\lambda_{\nu}^{s}
\end{equation}
in $\Omega$.
Using the fact that $\langle \id,\nu_x\rangle \in\R^{N}\otimes\eta(x)$ for $\mu^{s}$-almost every $x\in \Omega$, we get
\[
\langle \id,\nu_x^{\infty}\rangle\,\frac{d\lambda_{\nu}^{s}}{d\mu^{s}}=\frac{dD^{s}u}{d\mu^{s}}(x)-\langle \id,\nu_x\rangle= \xi(x)\otimes\eta(x)-\langle \id,\nu_x\rangle\,\in\,\R^{N}\otimes\eta(x)
\]
for $\mu^{s}$-almost every $x\in \Omega$.
Since $F^{\infty}$ is quasiconvex and 1-homogenous, we have $F^{\infty}(A)=(F^{\infty})^{c}(A)$ for all rank one $A\in\R^{N\times n}$, where the convex envelope is defined by
\[
G^{c}(A):=\sup\left\{H(A):\,H\textrm{ convex, } H\le G\right\},
\]
see \cite[Corollary 1.2]{KK}.
According to \cite[Lemma 5.5 (i)]{AB}, for any convex function $g:\R^{N\times n}\to \R$ we have
\begin{equation}\label{eq:consequence of convexity}
g(A_1+A_2)\le g(A_1)+g^{\infty}(A_2)
\end{equation}
for all $A_1,A_2\in\R^{N\times n}$. Note that in \eqref{eq:singular part of variation measure and Young measure wrt mu}, all three terms belong to $\R^N\otimes \eta(x)$ for $\mu^s$-almost every $x\in\Omega$. Since $\xi\mapsto F(\xi\otimes\eta(x))$ is convex for a fixed $x\in\Omega$ by the rank one convexity of $F$, we get by \eqref{eq:consequence of convexity}
\begin{align*}
F\left(\frac{dD^{s}u}{d\mu^{s}}(x)\right)
&\le F\left(\langle \id,\nu_x\rangle\right)+F^{\infty}\left(\langle \id, \nu_x^{\infty}\rangle \right)\frac{d\lambda_{\nu}^{s}}{d\mu^{s}}\\
&= F\left(\langle \id,\nu_x\rangle\right)+(F^{\infty})^{c}\left(\langle \id, \nu_x^{\infty}\rangle \right)\frac{d\lambda_{\nu}^{s}}{d\mu^{s}}\\
&\overset{Jensen}{\le} \langle F,\nu_x\rangle+\langle (F^{\infty})^{c}, \nu_x^{\infty}\rangle \frac{d\lambda_{\nu}^{s}}{d\mu^{s}}\\
&\le \langle F,\nu_x\rangle+\langle F^{\infty}, \nu_x^{\infty}\rangle \frac{d\lambda_{\nu}^{s}}{d\mu^{s}}
\end{align*}
for $\mu^s$-almost every $x\in\Omega$.
Combining this with \eqref{eq:Young measure wrt to mu} --- note that  $F\in \mathbf{SQ}(\R^{N\times n})\subset \mathbf{E}(\Omega;\R^{N\times n})$ --- we get for any ball $B(y,r)\subset\Omega$
\[
\int_{B(y,r)}F\left(\frac{dD^{s}u}{d\mu^{s}}\right)\,d\mu^{s}\le \liminf_{j\to\infty}\int_{B(y,r)}F\left(\frac{dDu_j}{d\mu}\right)\,d\mu.
\]
\end{proof}

\subsubsection{Combining the estimates}

Now we combine the previous two lemmas to prove Proposition \ref{prop:estimate from below}.

\begin{proof}[Proof of Proposition \ref{prop:estimate from below}.]
We keep assuming that $F\in \mathbf{SQ}(\R^{N\times n})$ with parameters $i,r_i\ge 1$ and linear growth $0\le F(A)\le M(1+|A|)$. Let $(u_j)\subset W_{\mu}^{1,1}(\Omega;\R^N)$ be a sequence with $u_j\to u$ in $L^1(\Omega;\R^N)$, and we can also assume that \eqref{eq:finiteness of sequence in lower semicontinuity} holds, so that the assumptions of both Lemma \ref{lem:estimate from below} and Lemma \ref{lem:estimate from below for singular part of mu} are satisfied.

Fix $\eps>0$. 
Let $H\subset\Omega$ be a Borel set with $\mathcal L^{n}(H)=0$ and $\mu^{s}(\Omega\setminus H)=0$. Also, let $D\subset \Omega$ be a Borel set with $\mu(D)=0$ and $|D^{s,\mu}u|(\Omega\setminus D)=0$. Take an open set $G\subset \Omega$ with $G\supset H\setminus D$ and
\begin{equation}\label{eq:conditions on G}
|D^{s,\mu}u|(G)+ \int_G  M(a+|\nabla u|)\,d\mathcal L^n<\eps.
\end{equation}
Consider the fine cover $\{\overline{B}(x,R)\}_{x\in H\setminus D}$ of the set $H\setminus D$, with the balls $\overline{B}(x,R)$ contained in $G$ and satisfying $\mu^{s}(\partial B(x,R))=0$.
By Vitali's covering theorem, we can pick a countable, disjoint collection $\{B_i\}_{i\in\N}:=\{B(x_i,R_i)\}_{i\in\N}$ with 
\begin{equation}\label{eq:singular part of mu covered by balls}
\mu^{s}\left((H\setminus D)\setminus \bigcup_{i=1}^{\infty}B_i\right)=0 \quad \textrm{and thus}\quad \mu^{s}\left(\Omega\setminus \bigcup_{i=1}^{\infty}B_i\right)=0.
\end{equation}
Pick also $m\in\N$ such that
\begin{equation}\label{eq:truncate union of balls}
\mu \left(\bigcup_{i=m}^{\infty}\overline{B}_i\right)+ M\int_{\bigcup_{i=m}^{\infty}\overline{B}_i}\left(1+\frac{d|D^{s}u|}{d\mu^{s}}\right)\,d\mu^{s}<\eps.
\end{equation}
By \eqref{eq:lower semicontinuity for singular part of mu} we have
\begin{equation}\label{eq:estimate from below near singular part of mu}
\begin{split}
\liminf_{j\to\infty} \int_{\bigcup_{i=1}^{m}B_i} F\left(\frac{dD u_j}{d\mu}\right)\,d\mu
&\ge \int_{\bigcup_{i=1}^{m}B_i}F\left(\frac{dD^s u}{d\mu^{s}}\right)d\mu^{s}\\
&\ge \int_{\Omega}F\left(\frac{dD^s u}{d\mu^{s}}\right)d\mu^{s}-\eps
\end{split}
\end{equation}
by \eqref{eq:truncate union of balls} and the linear growth of $F$.

By combining  \eqref{eq:singular part of mu covered by balls} and \eqref{eq:truncate union of balls}, we get
\begin{equation}\label{eq:singular part of mu almost covered by balls}
\mu^{s}\left(\Omega\setminus\bigcup_{i=1}^{m}\overline{B}_i\right)<\eps.
\end{equation}
Moreover, we can write \eqref{eq:estimate from below where mua lives} with the choice $U=\Omega\setminus \bigcup_{i=1}^{m}\overline{B_i}$:
\[
\begin{split}
 &\liminf_{j\to\infty} \int_{\Omega\setminus\bigcup_{i=1}^{m}\overline{B}_i} F\left(\frac{dDu_j}{d\mu}\right)\,d\mu \\
 &\quad\ge \int_{\Omega\setminus\bigcup_{i=1}^{m}\overline{B}_i} F\left(\frac{\nabla u}{a}\right) a\,d\mathcal L^{n}+\int_{\Omega\setminus\bigcup_{i=1}^{m}\overline{B}_i}F^{\infty}\left(\frac{dD^{s}u}{d|D^{s}u|}\right)d|D^{s}u|\\
 &\qquad\quad -(Mr_i+i)\mu^{s}\left(\Omega\setminus\bigcup_{i=1}^{m}\overline{B}_i\right)\\
 &\quad\overset{\eqref{eq:singular part of mu almost covered by balls}}{\ge} \int_{\Omega\setminus\bigcup_{i=1}^{m}\overline{B}_i} F\left(\frac{\nabla u}{a}\right) a\,d\mathcal L^{n}+\int_{\Omega\setminus\bigcup_{i=1}^{m}\overline{B}_i}F^{\infty}\left(\frac{dD^{s,\mu}u}{d|D^{s,\mu}u|}\right)d|D^{s,\mu}u|\\
&\qquad\quad-(Mr_i+i)\eps\\
&\quad\overset{\eqref{eq:conditions on G}}{\ge} \int_{\Omega\setminus\bigcup_{i=1}^{m}\overline{B}_i} F\left(\frac{\nabla u}{a}\right) a\,d\mathcal L^{n}+\int_{\Omega}F^{\infty}\left(\frac{dD^{s,\mu}u}{d|D^{s,\mu}u|}\right)d|D^{s,\mu}u|\\
&\qquad\quad -M\eps-(Mr_i+i)\eps\\
&\quad\overset{\eqref{eq:conditions on G}}{\ge} \int_{\Omega} F\left(\frac{\nabla u}{a}\right) a\,d\mathcal L^{n}+\int_{\Omega}F^{\infty}\left(\frac{dD^{s,\mu}u}{d|D^{s,\mu}u|}\right)d|D^{s,\mu}u|\\
&\qquad\quad-\eps-M\eps-(Mr_i+i)\eps.
\end{split}
\]
Combining this with \eqref{eq:estimate from below near singular part of mu}, we get
\begin{align*}
&\liminf_{j\to\infty} \int_{\Omega} F\left(\frac{dDu_j}{d\mu}\right)\,d\mu \\
&\qquad\qquad \ge \int_{\Omega} F\left(\frac{\nabla u}{a}\right) a\,d\mathcal L^{n}+\int_{\Omega}F^{\infty}\left(\frac{dD^{s,\mu}u}{d|D^{s,\mu}u|}\right)d|D^{s,\mu}u| \\
&\qquad\qquad\qquad+\int_{\bigcup_{i=1}^{m}B_i}F\left(\frac{dD^s u}{d\mu^{s}}\right)d\mu^{s}-3(Mr_i+i)\eps\\
&\qquad\qquad \overset{\eqref{eq:truncate union of balls}}{\ge} \int_{\Omega} F\left(\frac{\nabla u}{a}\right) a\,d\mathcal L^{n}+\int_{\Omega}F^{\infty}\left(\frac{dD^{s,\mu}u}{d|D^{s,\mu}u|}\right)d|D^{s,\mu}u| \\
&\qquad\qquad\qquad+\int_{\Omega}F\left(\frac{dD^s u}{d\mu^{s}}\right)d\mu^{s}-4(Mr_i+i)\eps.
\end{align*}
By letting $\eps\to 0$, we get the estimate from below.

Finally, we remove the assumption  $F\in \mathbf{SQ}(\R^{N\times n})$.
By \cite[Lemma 6.3]{KK}, we can find a sequence $F_i\in \mathbf{SQ}(\R^{N\times n})$ with $F_i(A)\searrow F(A)$ and $F_i^{\infty}(A)\searrow F^{\infty}(A)$ pointwise for every $A\in \R^{N\times n}$ as $i\to\infty$, and by making $M$ slightly larger, if necessary, we can also assume that  $m|A|\le F_i(A)\le M(1+|A|)$ for every $i\in\N$.

As before, let $(u_j)\subset W^{1,1}_{\mu}(\Omega;\R^N)$ with $u_j\to u$ in $L^1(\Omega;\R^N)$. We can again assume that \eqref{eq:finiteness of sequence in lower semicontinuity} holds, and by the coercivity $m|A|\le F(A)$, this implies that $(u_j)$ is a norm-bounded sequence in $\BV(\Omega;\R^N)$.
Thus by Theorem \ref{thm:generation of generalized Young measures for general mu}, a subsequence of $Du_j$ (not relabeled) generates a generalized Young measure $(\nu_x,\lambda_{\nu},\nu_x^{\infty})$ with respect to $\mu$. Thus we have for any $i\in\N$
\begin{equation}\label{eq:approximation step 1}
\begin{split}
&\int_{\Omega}F\left(\frac{dDu}{d\mu}\right)\,d\mu+\int_{\Omega}F^{\infty}
\left(\frac{dD^{s,\mu}u}{d|D^{s,\mu}u|}\right)\,d|D^{s,\mu}u|\\
&\qquad\qquad  \le \int_{\Omega}F_i\left(\frac{dDu}{d\mu}\right)\,d\mu+\int_{\Omega}F_i^{\infty}
\left(\frac{dD^{s,\mu}u}{d|D^{s,\mu}u|}\right)\,d|D^{s,\mu}u|\\
&\qquad\qquad  \le \liminf_{j\to\infty}\int_{\Omega}F_i\left(\frac{dDu_j}{d\mu}\right)\,d\mu\\
&\qquad\qquad  =\int_{\Omega}\langle F_i,\nu_x\rangle\,d\mu+\int_{\overline{\Omega}}\langle F_i^{\infty}, \nu_x^{\infty}\rangle\,d\lambda_{\nu}.
\end{split}
\end{equation}
On the other hand, by Lebesgue's dominated convergence theorem, as well as the fact that $F\in \mathbf{Q}(\R^{N\times n})\cap \mathbf{R}(\Omega;\R^{N\times n})\subset \mathbf{E}(\Omega;\R^{N\times n})$,
\begin{equation}\label{eq:approximation step 2}
\begin{split}
\lim_{i\to\infty}\left(\int_{\Omega}\langle F_i,\nu_x\rangle\,d\mu+\int_{\overline{\Omega}}\langle F_i^{\infty}, \nu_x^{\infty}\rangle\,d\lambda_{\nu}\right)
&= \int_{\Omega}\langle F,\nu_x\rangle\,d\mu+\int_{\overline{\Omega}}\langle F^{\infty}, \nu_x^{\infty}\rangle\,d\lambda_{\nu}\\
&=\lim_{j\to\infty}\int_{\Omega}F\left(\frac{dDu_j}{d\mu}\right)\,d\mu.
\end{split}
\end{equation}
By combining \eqref{eq:approximation step 1} and \eqref{eq:approximation step 2}, we get the desired estimate from below.

\end{proof}

\subsection{Estimate from above}

Recall from \eqref{eq:definition by relaxation} the definition of the functional $\mathcal F_{*}$ by relaxation. We prove that the estimate from above holds for the integral representation of $\mathcal F_*$. Here our proof is not based on the theory of Young measures, so we can allow for somewhat weaker assumptions on $F$.

\begin{proposition}\label{prop:estimate from above}
Let $\Omega\subset\R^{n}$ be a bounded open set, let $\mu\in\mathcal M^+(\Omega)$ with $\mathcal L^{n}\ll \mu$, let $F\in \mathbf{Q}(\R^{N\times n})$ with
\[
0\le F(A)\le M(1+|A|),\quad A\in\R^{N\times n}
\]
for some $M\ge 1$, and let $u\in\BV(\Omega;\R^N)$. Then we have
\[
\mathcal F_{*}(u,\Omega)\le \int_{\Omega}F\left(\frac{dDu}{d\mu}\right)\,d\mu+\int_{\Omega} F^{\infty}\left(\frac{dD^{s,\mu}u}{d|D^{s,\mu}u|}\right)\,d|D^{s,\mu}u|.
\]
\end{proposition}

\begin{proof}
Again, by Definition \ref{def:special quasiconvex integrands} and \cite[Lemma 6.3]{KK} we can find a sequence $F_i\in \mathbf{SQ}(\R^{N\times n})$ with parameters $i\in\N$, $r_i>0$ such that $F_i(A)\searrow F(A)$ and $F_i^{\infty}(A)\searrow F^{\infty}(A)$ pointwise for every $A\in \R^{N\times n}$ as $i\to\infty$. Moreover, by making $M$ slightly larger, if necessary, we have that $0\le F_i(A)\le M(1+|A|)$ for all $i\in\N$ and $A\in\R^{N\times n}$.
Fix $i\in\N$.

The proof is based on mollifying the function $u$ in a small set.
Take a Borel set $D\subset\Omega$ with $|D^{s,\mu}u|(\Omega\setminus D)=0$ and $\mu(D)=0$.
Then take an open set $G\supset D$ with $\mathcal L^n(G)$ and $\mu(G)$ so small that
\begin{equation}\label{eq:properties of the set G}
\int_{G} M(1+|\nabla u|)\,d\mathcal L^n+\int_{G} M\left|\frac{d D^s u}{d \mu}\right|\,d\mu
+M(r_i+i)\mathcal L^n(G)+M(1+r_i)\mu(G)
\end{equation}
is less than $1/i$; this is possible by the absolute continuity of integrals.
By Lemma \ref{lem:smooth area-strict approximation of BV} we can pick a sequence $(v_j)\subset \BV_u(G;\R^N)\cap C^{\infty}(G;\R^N)$ (note boundary values) that converges to $u$ $\langle\cdot\rangle$-strictly in $\BV(G;\R^N)$. Fix also $j\in\N$.

Using the linear growth of $F_i$, we estimate
\begin{equation*}
\begin{split}
 \int_{G} F\left(\frac{dDv_j}{d\mu}\right)\,d\mu
&\le  \int_{G} F_i\left(\frac{dDv_j}{d\mu}\right)\,d\mu
= \int_{G} F_i\left(\frac{\nabla v_j}{a}\right)a\,d\mathcal L^n\\
&\le \int_{G\cap\{|\nabla v_j/a|>r_i\}} F_i\left(\frac{\nabla v_j}{a}\right)a\,d\mathcal L^n+M(1+r_i)\mu(G),
\end{split}
\end{equation*}
where by the fact that $F(A)=F^{\infty}(A)-i$ for $|A|\ge r_i$, the last integral equals
\begin{equation*}
\begin{split}
  \int_{G\cap\{|\nabla v_j/a|>r_i\}} &\left(F_i^{\infty}\left(\frac{\nabla v_j}{a}\right)  -i\right)  a\,d\mathcal L^n
 \le  \int_{G\cap\{|\nabla v_j/a|>r_i\}} F_i^{\infty}\left(\frac{\nabla v_j}{a}\right)a\,d\mathcal L^n\\
 &\qquad\qquad\quad =  \int_{G\cap\{|\nabla v_j/a|>r_i\}} F_i^{\infty}(\nabla v_j)\,d\mathcal L^n\\
 &\qquad\qquad\quad\le  \int_{G} F_i^{\infty}\left(\nabla v_j\right)\,d\mathcal L^n\\
 &\qquad\qquad\quad\le  \int_{G\cap\{|\nabla v_j|>r_i\}} F_i^{\infty}\left(\nabla v_j\right)\,d\mathcal L^n+Mr_i\mathcal L^n(G)\\
 &\qquad\qquad\quad=  \int_{G\cap\{|\nabla v_j|>r_i\}} (F_i\left(\nabla v_j\right)+i)\,d\mathcal L^n+Mr_i\mathcal L^n(G)\\
 &\qquad\qquad\quad\le  \int_{G} F_i\left(\nabla v_j\right)\,d\mathcal L^n+M(r_i+i)\mathcal L^n(G).
\end{split}
\end{equation*}
Now, since $F_i\in\mathbf{SQ}(\R^{N\times n})\subset \mathbf{E}(G;\R^{N\times n})$ (constant in the $x$-variable) and $v_j\to u$ $\langle\cdot\rangle$-strictly in $\BV(G;\R^N)$, we can apply Reshetnyak's continuity theorem, Theorem \ref{thm:Reshetnyak}, to obtain
\begin{align*}
&\liminf_{j\to\infty} \int_{G} F\left(\frac{dDv_j}{d\mu}\right)\,d\mu\\
& \qquad \le \liminf_{j\to\infty}\int_{G} F_i\left(\nabla v_j\right)\,d\mathcal L^n
+M(r_i+i)\mathcal L^n(G)+M(1+r_i)\mu(G)\\
& \qquad= \int_{G} F_i(\nabla u)\,d\mathcal L^n+\int_{G} F_i^{\infty}\left(\frac{d D^{s} u}{d |D^{s} u|}\right)\,d|D^{s} u|\\
&\qquad\quad\ +M(r_i+i)\mathcal L^n(G)+M(1+r_i)\mu(G)\\
& \qquad\le \int_{G} M(1+|\nabla u|)\,d\mathcal L^n+\int_{G} M\left|\frac{d D^s u}{d \mu}\right|\,d\mu\\
&\qquad\quad\ +\int_{G} F_i^{\infty}\left(\frac{d D^{s,\mu} u}{d |D^{s\,\mu} u|}\right)\,d|D^{s,\mu} u|+M(r_i+i)\mathcal L^n(G)+M(1+r_i)\mu(G)\\
&\qquad\le \int_{G} F_i^{\infty}\left(\frac{d D^{s,\mu} u}{d |D^{s\,\mu} u|}\right)\,d|D^{s,\mu} u|+1/i
\end{align*}
by \eqref{eq:properties of the set G}.
Then define for each $j\in\N$
\[
u_{j}:=
\left\{ \begin{alignedat}{2}
& v_j\qquad && \textrm{in } G,\\
& u\qquad && \textrm{in } \Omega\setminus G.
 \end{alignedat}
\right.
\]
The fact that $v_j\in\BV_u(G;\R^N)$ implies by definition (given before Lemma \ref{lem:smooth area-strict approximation of BV}) that $Du_j=Du\mres \Omega\setminus G+Dv_j\mres G$.
Thus it is clear that $u_j\in W^{1,1}_{\mu}(\Omega; \R^{N})$, and also $u_j\to u$ in $L^{1}(\Omega;\R^N)$, so that $u_j$ is an admissible sequence for $\mathcal F_*(u,\Omega)$.
In total, we obtain
\begin{align*}
\mathcal F_{*}  (u,\Omega)
&\le \liminf_{j\to\infty} \int_{\Omega} F\left(\frac{dDu_j}{d\mu}\right)\,d\mu\\
&=\liminf_{j\to\infty} \int_{G} F\left(\frac{dDv_j}{d\mu}\right)\,d\mu+\int_{\Omega\setminus G} F\left(\frac{dDu}{d\mu}\right)\,d\mu\\
&\le \int_{G} F_i^{\infty}\left(\frac{d D^{s,\mu} u}{d |D^{s\,\mu} u|}\right)\,d|D^{s,\mu} u|+\int_{\Omega\setminus G} F\left(\frac{dDu}{d\mu}\right)\,d\mu+1/i\\
&\le \int_{\Omega} F_i\left(\frac{dDu}{d\mu}\right)\,d\mu+\int_{\Omega} F_i^{\infty}\left(\frac{d D^{s,\mu} u}{d |D^{s\,\mu} u|}\right)\,d|D^{s,\mu} u|+1/i.
\end{align*}
Letting $i\to\infty$, by Lebesgue's monotone or dominated convergence we get the desired estimate from above.
\end{proof}

\subsection{Some examples}
Let us briefly consider why it is necessary to assume that $\mathcal L^{n}\ll \mu$, at least in order to obtain the integral representation \eqref{eq:integral representation}.
The reason is that the estimate from above may be violated without this assumption. We note that the integral representation \eqref{eq:integral representation} always takes a value at most
\[
M\mu(\Omega)+M|Du|(\Omega),
\]
which is finite for a $\BV$ function $u\in\BV(\Omega;\R^N)$. On the other hand, if it is not true that $\mathcal L^{n}\ll \mu$, then there can be a large set not "seen" by the measure $\mu$, and as a result it may simply be impossible to approximate certain $\BV$ functions in the $L^{1}$-sense by functions in the class $W^{1,1}_{\mu}(\Omega;\R^N)$. Consider the following examples.

\begin{example}
Suppose that there is an \emph{open} set $B\subset\Omega$ (which we can assume to be a ball) with $\mu(B)=0$ but of course $\mathcal L^{n}(B)>0$. Take a nonconstant $u\in C^{1}_c(B)$, and note that all functions $u_j\in W^{1,1}_{\mu}(\Omega)$ satisfy $|Du_j|(B)=0$ and are thus constant in the ball $B$. Thus there is no sequence of functions $u_j\in W^{1,1}_{\mu}(\Omega)$ with $u_j\to u$ in $L^{1}(\Omega)$, and consequently $\mathcal F_{*}(u,\Omega)=\infty$.
\end{example}

Even if the support of $\mu$ is the whole of $\overline{\Omega}$, the estimate from above may fail.
\begin{example}
Take $\Omega$ to be the open unit square on the plane, and let $A\subset\Omega$ be a "fat" Sierpinski carpet, with $\mathcal L^{2}(A)=1/2$. Then define the weight $w=\mathbbm{1}_{\Omega\setminus A}$, and $\mu:=w\,\mathcal L^{2}$. Clearly the absolute continuity assumption $\mathcal L^2\ll \mu$ is violated, but the support of $\mu$ is the whole of $\overline{\Omega}$. By using the properties of $\BV$ functions restricted to lines, see e.g. \cite[Section 3.11]{AFP}, we obtain that any function $v\in W^{1,1}_{\mu}(\Omega)$ is constant almost everywhere in $A$. If we define a BV function $u\in\BV(\Omega)$ e.g. as  $u(x,y):=x$, there is no sequence $u_j\in W^{1,1}_{\mu}(\Omega)$ for which $u_j\to u$ in $L^{1}(\Omega)$, and consequently $\mathcal F_{*}(u,\Omega)=\infty$.
\end{example}

However, it is not clear to us whether the assumption $\mathcal L^n\ll \mu$, or the assumption on the integrand $F\in \mathbf{R}(\Omega;\R^{N\times n})$, are necessary in our main result, Theorem \ref{thm:lower semicontinuity result in intro}.


\section{The lower semicontinuity theorem}

From the integral representation, we obtain the following lower semicontinuity result.

\begin{proposition}\label{prop:lower semicontinuity}
Let $\Omega\subset \R^n$ be a bounded open set with $\mathcal L^n(\partial\Omega)=0$, let $\mu\in\mathcal M^+(\Omega)$ with $\mathcal L^n\ll \mu$, and let $F\in \mathbf{R}(\Omega;\R^{N\times n})\cap \mathbf{Q}(\R^{N\times n})$ with
\[
m|A|\le F(A)\le M(1+|A|)
\]
for some $0<m\le M$.
Then the functional
\[
\mathcal F(u):=\int_{\Omega}F\left(\frac{dDu}{d\mu}\right)\,d\mu+\int_{\Omega}F^{\infty}\left(\frac{dD^{s,\mu}u}{d|D^{s,\mu}u|}\right)\,d|D^{s,\mu}u|,\quad u\in\BV(\Omega;\R^N),
\]
is lower semicontinuous with respect to convergence in $L^1(\Omega;\R^N)$.
\end{proposition}

\begin{proof}
The relaxed functional $\mathcal F_*(u,\Omega)$ given in \eqref{eq:definition by relaxation} is obviously lower semicontinuous with respect to convergence in $L^1(\Omega;\R^N)$, and by Proposition \ref{prop:estimate from below} and Proposition \ref{prop:estimate from above} it equals the functional $\mathcal F(u)$ given in this proposition.
\end{proof}

We recall Jensen's inequalities for gradient Young measures with respect to the Lebesgue measure $\mathcal L^n$, given in Theorem \ref{thm:Jensens inequalities for generalized Young measures}.
We can now partially generalize these inequalities to the case of a general measure $\mu$.

\begin{theorem}\label{thm:Jensens inequalities wrt mu}
Let $\Omega\subset \R^n$ be a bounded open set with $\mathcal L^n(\partial\Omega)=0$, let $\mu\in\mathcal M^+(\Omega)$ with $\mathcal L^n\ll\mu$, let $F\in \mathbf{R}(\Omega;\R^{N\times n})\cap \mathbf{Q}(\R^{N\times n})$ with $F\ge 0$, let $u\in\BV(\Omega;\R^N)$, and let $\nu\in\mathbf{Y}(\Omega,\mu;\R^{N\times n})$ be a gradient Young measure with $\lambda_{\nu}(\partial\Omega)=0$ and with barycenter $Du$.
Then the following hold:
\begin{align}
\label{eq:Jensen wrt mu absolutely continuous part}& F\left(\frac{dDu}{d\mu}\right)\le \langle F,\nu_x\rangle+\langle F^{\infty},\nu_x^{\infty}\rangle\frac{d\lambda_{\nu}}{d\mu}(x)\quad\textrm{for }\mu\textrm{-almost every } x\in\Omega,\\
\label{eq:Jensen wrt mu singular part}& F^{\infty}\left(\frac{dD^{s,\mu} u}{d|D^{s,\mu} u|}\right)|D^{s,\mu} u|\le \langle F^{\infty},\nu_x^{\infty}\rangle\lambda^{s,\mu}_{\nu}\quad\textrm{as measures}.
\end{align}

\end{theorem}

\begin{proof}
Take a sequence $(u_j)\subset \BV(\Omega;\R^N)$ that generates $\nu$. We know that $u_j\overset{*}{\rightharpoondown}u$ in $\BV(\Omega;\R^N)$, see the discussion after \eqref{eq:generation of gradient Young measure wrt mu}.
Note that $F\in\mathbf{R}(\Omega;\R^{N\times n})\cap \mathbf{Q}(\R^{N\times n})\subset \mathbf{E}(\Omega;\R^{N\times n})$ (constant in the $x$-variable), so that $F$ necessarily has linear growth $F(A)\le M(1+|A|)$ for some $M\ge 0$.
Let us first also assume that $F$ has the coercivity property $m|A|\le F(A)$ for some $m>0$ and all $A\in\R^{N\times n}$.
By combining our lower semicontinuity result, Proposition \ref{prop:lower semicontinuity}, with the fact that $F\in \mathbf{E}(\Omega;\R^{N\times n})$, we obtain
\begin{align*}
&\int_{\Omega}F\left(\frac{dDu}{d\mu}\right)\,d\mu+\int_{\Omega}F^{\infty}\left(\frac{dD^{s,\mu}u}{d|D^{s,\mu}u|}\right)\,d|D^{s,\mu}u|\\
&\qquad\qquad\le\liminf_{j\to\infty}\left(\int_{\Omega} F\left(\frac{dDu_j}{d\mu}\right)\,d\mu
+\int_{\Omega}F^{\infty}\left(\frac{dD^{s,\mu}u_j}{d|D^{s,\mu}u_j|}\right)\,d|D^{s,\mu}u_j|\right)\\
&\qquad\qquad =\int_{\Omega} \langle F_i,\nu_x\rangle\, d\mu+
\int_{\Omega} \langle F_i^{\infty},\nu_x^{\infty}\rangle\,d\lambda_{\nu}.
\end{align*}
We can equally well write the above inequality in any open $U\subset\Omega$ (in particular, a ball) with $\lambda_{\nu}(\partial U)=0$. Thus we can differentiate the inequality with respect to $\mu$, and obtain \eqref{eq:Jensen wrt mu absolutely continuous part} by the Besicovitch differentiation theorem (see e.g. \cite[Theorem 2.22]{AFP}). By writing the above inequality for balls from a suitable Vitali covering of $\Omega$, we obtain \eqref{eq:Jensen wrt mu singular part}.

The general case can be obtained by writing \eqref{eq:Jensen wrt mu absolutely continuous part} and \eqref{eq:Jensen wrt mu singular part} for integrands
\[
F_i(A):=\max\{F(A),|A|/i\},\quad i\in\N,
\]
and letting $i\to\infty$.
\end{proof}

\begin{corollary}\label{cor:Jensens inequalities wrt mu}
With $\Omega$, $\mu$, $u$, and $\nu$ as in the previous theorem, there exist sets $E_1,E_2\subset \Omega$ with $\mu(E_1)=0$ and $| D^{s,\mu}u|(E_2)=0$ such that for every $F\in \mathbf{R}(\Omega;\R^{N\times n})\cap \mathbf{Q}(\R^{N\times n})$ with $F\ge 0$, we have
\begin{align}
\label{eq:Jensen wrt mu absolutely continuous part revisited}F\left(\frac{dDu}{d\mu}\right)\le \langle F,\nu_x\rangle+\langle F^{\infty},\nu_x^{\infty}\rangle\frac{d\lambda_{\nu}}{d\mu}(x)\quad\textrm{for every } x\in\Omega\setminus E_1,\\
\label{eq:Jensen wrt mu singular part revisited}F^{\infty}\left(\frac{dD^{s,\mu} u}{d|D^{s,\mu} u|}\right)\le \langle F^{\infty},\nu_x^{\infty}\rangle\frac{d\lambda^{s,\mu}_{\nu}}{d|D^{s,\mu} u|}\quad\textrm{for every } x\in\Omega\setminus E_2.
\end{align}
\end{corollary}
The point is that we can find exceptional sets that do not depend on the integrand $F$.
\begin{proof}
Again, we note that $\mathbf{R}(\Omega;\R^{N\times n})\cap \mathbf{Q}(\R^{N\times n})\subset \mathbf{E}(\Omega;\R^{N\times n})$ (constant in the $x$-variable). Recalling the transformation $T$ given in Section \ref{sec:generalized Young measures}, we have that
\[
\{T(F):\,F\in \mathbf{R}(\Omega;\R^{N\times n})\cap \mathbf{Q}(\R^{N\times n}),\,F\ge 0\}
\]
contains a countable dense subset $\{G_i\}_{i\in\N}$, since it is contained in the separable space
$C(\overline{\mathbb{B}^{N\times m}})$. Then \eqref{eq:Jensen wrt mu absolutely continuous part revisited} and \eqref{eq:Jensen wrt mu singular part revisited} hold for some choice of sets $E_1,E_2\subset \Omega$ with $\mu(E_1)=0$ and $| D^{s,\mu}u|(E_2)=0$, and with $F=T^{-1}G_i$ for any $i\in\N$. It is easy to see for any $F\in \mathbf{R}(\Omega;\R^{N\times n})\cap \mathbf{Q}(\R^{N\times n})$, $F\ge 0$ that
\[
F_k\left(\frac{dDu}{d\mu}(x)\right)-F\left(\frac{dDu}{d\mu}(x)\right)\to 0
\]
for every $x\in\Omega\setminus E_1$, for a sequence $(F_k)\subset \{T^{-1}G_i\}_{i\in\N}$ with $T(F_k)\to T(F)$ in $C(\overline{\mathbb{B}^{N\times m}})$. The other terms are handled similarly, and so we get the desired inequalities.
\end{proof}

Now we can prove our semicontinuity result, where we also allow for $x$-dependence of the integrand. The result could also be given without a boundary term, but its inclusion simplifies our proof. In the case $\mu=\mathcal L^n$, an analogous result was given in \cite[Theorem 10]{KR2}.

\begin{theorem}\label{thm:lower semicontinuity for $x$-dependent integrands}
Let $\Omega\subset \R^n$ be a bounded Lipschitz domain with inner boundary normal $\nu_{\Omega}$,
let $\mu\in\mathcal M^+(\Omega)$ with $\mathcal L^n\ll\mu$, and let $F\in \mathbf{R}(\Omega;\R^{N\times n})$ be nonnegative and $\mu\times\mathcal B(\R^{N\times n})$-measurable such that $A\mapsto F(x,A)$ is quasiconvex for each fixed $x\in\overline{\Omega}$. Then the functional
\begin{align*}
\mathcal F(u):=&\int_{\Omega}F\left(x,\frac{dDu}{d\mu}\right)\,d\mu+\int_{\Omega}F^{\infty}\left(x,\frac{dD^{s,\mu}u}{d|D^{s,\mu}u|}\right)\,d|D^{s,\mu}u|\\
&+\int_{\partial\Omega}F^{\infty}\left(x,\frac{u}{|u|}\otimes \nu_{\Omega}\right)|u|\,d\mathcal H^{n-1}
\end{align*}
is weakly* sequentially lower semicontinuous in $\BV(\Omega;\R^N)$.
\end{theorem}

Note that in the last term, $u$ is a \emph{boundary trace}, see e.g. \cite[Section 3.7]{AFP}.

\begin{proof}
Let $u_j\overset{*}{\rightharpoondown} u$ in $\BV(\Omega;\R^N)$. Take a bounded Lipschitz domain $\Omega'\Supset \Omega$, and denote by $u_j^e,u^e$ the zero extensions of $u_j,u$ to $\Omega'\setminus \Omega$. Since $\Omega$ is a bounded Lipschitz domain, we can use standard gluing theorems for $\BV$ functions, see e.g. \cite[Proposition 3.21, Theorem 3.84, Theorem 3.86]{AFP}, to obtain that $u_j^e\in\BV(\Omega';\R^N)$ with
\[
Du_j^e= \nabla u_j\,\mathcal L^n \mres\Omega + D^s u_j+u_j\otimes\nu_{\Omega}\,\mathcal H^{n-1}\mres\partial\Omega
\]
and $\Vert u_j\Vert_{L^1(\partial\Omega;\R^N)}\le C\Vert u_j\Vert_{\BV(\Omega;\R^N)}$ with $C$ depending only on $\Omega$; and similarly for $u^e$. By the weak* convergence, $u_j$ is a norm-bounded sequence in $\BV(\Omega;\R^N)$,
so we have that $u_j^e$ is a norm-bounded sequence in $\BV(\Omega';\R^N)$ and that $u_j^e\to u^e$ in $L^1(\Omega';\R^N)$. This implies that $u_j^e\overset{*}{\rightharpoondown} u^e$ in $\BV(\Omega';\R^N)$.

Since $F\in\mathbf{R}(\Omega,\R^{N\times n})$, $F^{\infty}(x,A)$ is continuous on $\overline{\Omega}\times \partial\mathbb{B}^{N\times n}$, which is a compact set. By the Tietze extension theorem, we can extend $F^{\infty}$ to $\overline{\Omega'}\times \partial\mathbb{B}^{N\times n}$ as a continuous
nonnegative function $(F^e)^{\infty}$. If we define $F^e(x,tA):=t(F^e)^{\infty}(x,A)$ for any $t\ge 0$, $A\in\R^{N\times n}$, and $x\in\overline{\Omega'}$, we see that our notation is consistent in that the recession function of $F^e$ is indeed $(F^e)^{\infty}$.
We also extend $\mu$ by $\mu^e:=\mu\mres \Omega+\mathcal L^N\mres (\R^n\setminus \Omega)$.
Then we see that $F^e\in \mathbf{R}(\Omega';\R^{N\times n})$ is nonnegative and $\mu^e\times \mathcal B({\R^{N\times n}})$-measurable. We write
\begin{align*}
\mathcal F^e(u_j^e):&=\int_{\Omega'}F^e\left(x,\frac{dDu_j^e}{d\mu^e}\right)\,d\mu^e+\int_{\Omega'}(F^e)^{\infty}\left(x,\frac{dD^{s,\mu^e}u_j^e}{d|D^{s,\mu^e}u_j^e|}\right)\,d|D^{s,\mu^e}u_j^e|\\
&=\int_{\Omega}F\left(x,\frac{dDu_j}{d\mu}\right)\,d\mu+\int_{\Omega}F^{\infty}\left(x,\frac{dD^{s,\mu}u_j}{d|D^{s,\mu}u_j|}\right)\,d|D^{s,\mu}u_j|\\
&\qquad +\int_{\partial\Omega}F^{\infty}\left(x,\frac{u_j}{|u_j|}\otimes \nu_{\Omega}\right)|u_j|\,d\mathcal H^{n-1}\\
&=\mathcal F(u_j),
\end{align*}
and similarly for $u^e$.
We conclude that we need to prove that $\mathcal F^e(u^e)\le \liminf_{j\to\infty}\mathcal F^e(u_j^e)$.
Pick first a subsequence (not relabeled) that gives this limit, and then by Theorem \ref{thm:generation of generalized Young measures for general mu} and Corollary \ref{cor:representation wrt mu} we can pick a further subsequence (not relabeled) such that the sequence $Du_j^e$ generates a generalized Young measure $\nu=(\nu_x,\lambda_{\nu},\nu_x^{\infty})$, with respect to $\mu^e$. Clearly $\lambda_{\nu}(\partial\Omega')=0$, and then the barycenter of $\nu$ is $Du^e$, see the discussion after \eqref{eq:generation of gradient Young measure wrt mu}. Note that for any fixed $x\in\overline{\Omega}$, $F^e(x,\cdot)\in\mathbf{R}(\Omega, \R^{N\times n})\cap \mathbf{Q}(\R^{N\times n})$, so that we can apply Corollary \ref{cor:Jensens inequalities wrt mu} to obtain
\begin{align*}
&\liminf_{j\to\infty}\mathcal F^e(u_j^e)
= \int_{\Omega'}\langle F^e(x,\cdot),\nu_x\rangle\,d\mu^e+\int_{\Omega'}\langle(F^e)^{\infty}(x,\cdot),\nu_x^{\infty}\rangle\,d\lambda_{\nu}(x)\\
&\qquad\ge \int_{\Omega'}F^e\left(x,\frac{dDu^e}{d\mu^e}\right)\,d\mu^e+\int_{\Omega'}(F^e)^{\infty}\left(x,\frac{dD^{s,\mu^e} u^e}{d|D^{s,\mu^e} u^e|}\right)\,d|D^{s,\mu^e} u^e|\\
&\qquad = \mathcal F^e(u^e).
\end{align*}
\end{proof}



\end{document}